\documentclass[a4paper,11pt]{article}

\usepackage[T1]{fontenc}
\usepackage{lmodern}

\usepackage{dsfont}

\usepackage[latin1]{inputenc}
\usepackage{amsmath}
\usepackage{amsthm}
\usepackage{amssymb}
\usepackage{mathrsfs}
\usepackage{graphicx}
\usepackage[all]{xy}
\usepackage{hyperref}

\usepackage{stmaryrd}
\usepackage{caption}

\usepackage{abstract} 

\newtheorem{thm}{Theorem}[section]
\newtheorem{cor}[thm]{Corollary}
\newtheorem{claim}[thm]{Claim}
\newtheorem{fact}[thm]{Fact}

\newtheorem{lemma}[thm]{Lemma}
\newtheorem{prop}[thm]{Proposition}

\theoremstyle{definition}
\newtheorem{definition}[thm]{Definition}

\newtheorem{remark}[thm]{Remark}
\newtheorem{question}[thm]{Question}
\newtheorem{problem}[thm]{Problem}
\newtheorem{conj}[thm]{Conjecture}

\usepackage{enumitem}

\title{Cactus groups from the viewpoint of geometric group theory}
\date{\today}
\author{Anthony Genevois}

\begin{document}

\maketitle

\begin{abstract}
Cactus groups and their pure subgroups appear in various fields of mathematics and are currently attracting attention from diverse mathematical communities. They share similarities with both right-angled Coxeter groups and braid groups. In this article, our goal is to highlight the tools offered by geometric group theory for the study of these groups. Among the new contributions made possible thanks to this geometric perspective, we describe an explicit and efficient solution to the conjugacy problem, and we prove that cactus groups are virtually cocompact special and acylindrically hyperbolic. 
\end{abstract}

\tableofcontents

\newpage
\section{Introduction}

\noindent
Given an integer $n \geq 2$, the \emph{cactus group} $J_n$ can be defined by the generators $s_{p,q}$, with $1 \leq p < q \leq n$, submitted to the following relations:
\begin{itemize}
	\item $s_{p,q}^2=1$ for every $1 \leq p<q \leq n$;
	\item $s_{p,q}s_{m,r} = s_{m,r}s_{p,q}$  for all $1 \leq p<q \leq n$ and $1 \leq m < r \leq n$ satisfying $[p,q] \cap [m,r] = \emptyset$;
	\item $s_{p,q}s_{m,r} = s_{p+q-r,p+q-m}s_{p,q}$ for all $1 \leq p<q \leq n$ and $1 \leq m < r \leq n$ satisfying $[m,r] \subset [p,q]$.
\end{itemize}
The presentation is close to presentations of right-angled Coxeter groups, but the third item introduces \emph{mock commutations} which distinguish cactus groups from such Coxeter groups.

\medskip \noindent
\begin{minipage}{0.3\linewidth}
\includegraphics[width=0.95\linewidth]{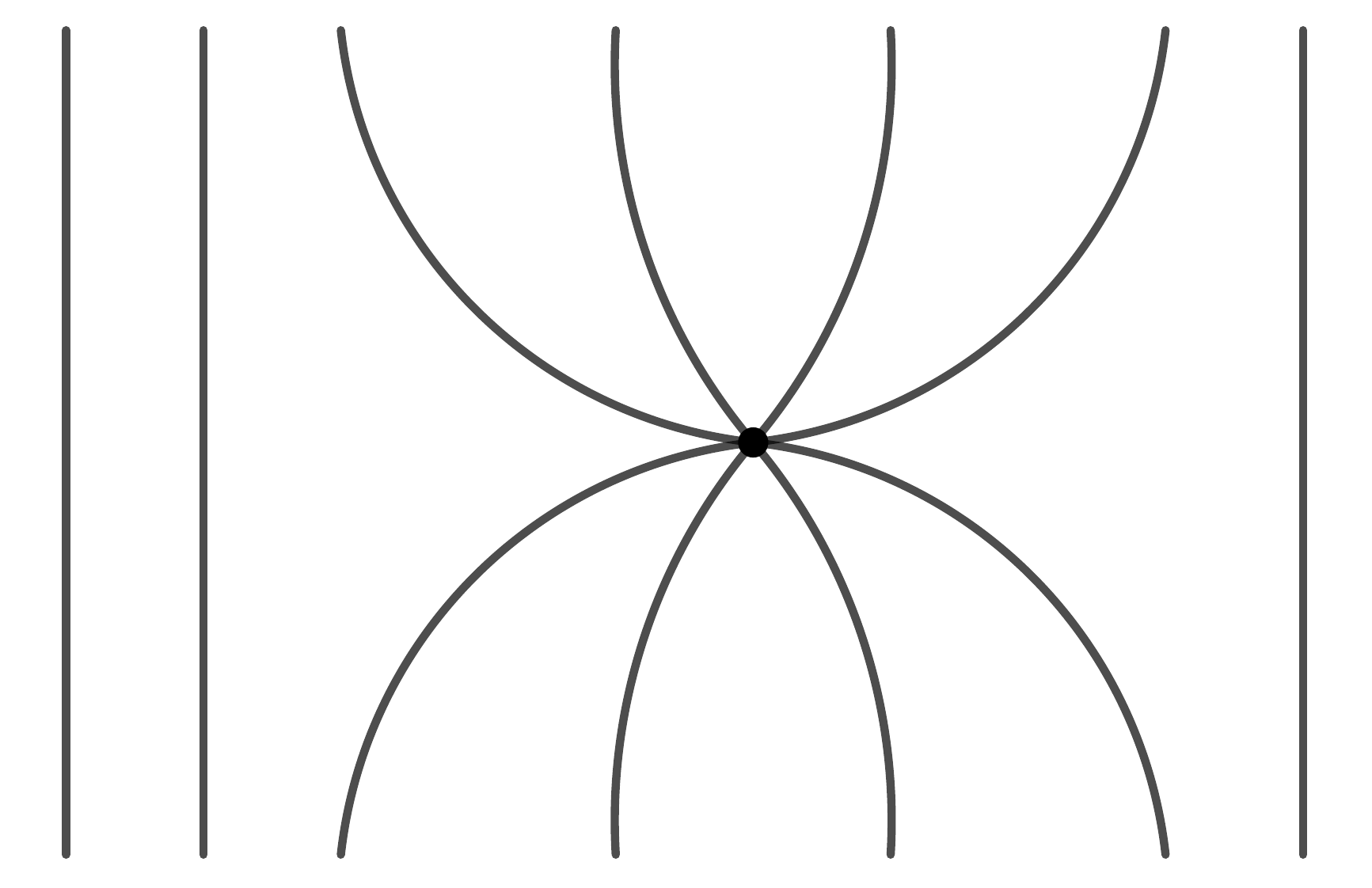}
\end{minipage}
\begin{minipage}{0.68\linewidth}
More visually, the elements of $J_n$ can represented by \emph{braided-like pictures}. Given $n$ vertical strands and two indices $1 \leq p<q \leq n$, the generator $s_{p,q}$ centrally inverts the strands numbered from $p$ to $q$, all the strands meeting simultaneously at a unique node during the process. For instance, the picture on the left illustrates the generator $s_{3,6}$ in the cactus group $J_7$. 
\end{minipage}

\medskip \noindent
Then, the relations given above can be visualised as follows:

\medskip \noindent
\includegraphics[width=\linewidth]{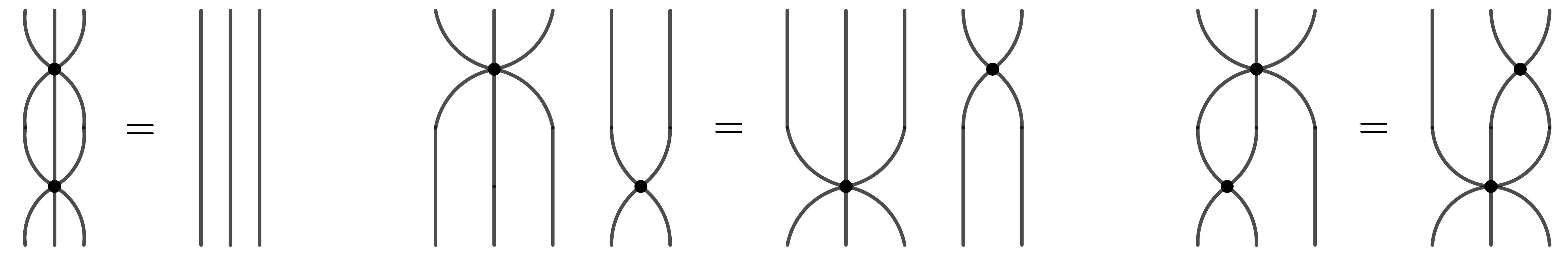}

\medskip \noindent
In this article, we use the convention that, given two elements $g,h \in J_n$, the product $gh$ is represented by braided-like picture obtained by concatenating a picture representing $h$ below a picture representing $g$. Thus, one obtains a natural morphism
$$\Sigma : J_n \to S_n$$
to the symmetric group $S_n$ as follows. Number from left to right the top and bottom endpoints of our $n$ strands from $1$ to $n$. Then, given an element $g \in J_n$ represented by a braided-like picture, the permutation $\Sigma(g)$ is defined by sending every $1 \leq i \leq n$ to the top number of the strands having $i$ as its bottom number\footnote{It is possible to define $\Sigma(g)$ from top to bottom instead of bottom to top, but then we have to modify our convention and represent a product $gh$ in $J_n$ by a picture obtained by concatenation a picture representing $h$ above a picture representing $g$.}. The kernel of $\Sigma$ is referred to as the \emph{pure cactus group} $PJ_n$. 

\medskip \noindent
Cactus groups and their pure subgroups appear in various fields of mathematics and are currently attracting attention from diverse mathematical communities. In this article, our goal is to highlight the tools offered by geometric group theory for the study of these groups. As an illustration of the relevancy of this point of view, we extract valuable information about cactus groups from our geometric perspective. This allows us to obtain new results but also to shed a new light on results previous obtained.

\paragraph{Median geometry.} The key property that underlies our geometric study of cactus groups is that the canonical Cayley graphs of these graphs are \emph{median}, which immediately implies that cactus groups act properly on cocompactly on median graphs.

\begin{thm}\label{thm:IntroMedian}
For every $n \geq 2$, the Cayley graph $\mathscr{C}_n$ of $J_n$ with respect to $\{s_{p,q} \mid 1 \leq p<q \leq n\}$ is median.
\end{thm}

\noindent
A connected graph $X$ is \emph{median} if, for all vertices $x_1,x_2,x_3 \in X$, there exists a unique vertex $m \in X$, referred to as the \emph{median point} of $x_1,x_2,x_3$, which satisfies
$$d(x_i,x_j)=d(x_i,m)+d(m,x_j) \text{ for all } i \neq j.$$
Simplicial trees are the among the simplest examples, where the median point of three vertices is given by the centre of the tripod they delimit. It is worth mentioning that median graphs coincide with one-skeleta of \emph{CAT(0) cube complexes}, and that the latter terminology is more common in the literature (despite the fact that CAT(0) cube complexes are neither thought of nor studied via their CAT(0) geometry but via their median geometry). Theorem~\ref{thm:IntroMedian} can be deduced from \cite{MR1668119} and \cite{MR2395169}. We give a direct proof in Section~\ref{section:CactusMedian}. 

\medskip \noindent
Knowing that a group admits a proper and cocompact action on a median graph automatically implies various properties, including:

\begin{cor}
The following assertions hold.
\begin{itemize}[itemsep=0pt]
	\item Cactus groups are biautomatic, which implies that:
\begin{itemize}[itemsep=0pt]
	\item their word and conjugacy problems are solvable;
	\item they satisfy a quadratic isoperimetric inequality;
	\item they contain only finitely many conjugacy classes of finite subgroups;
	\item every infinite-order element admits only finitely many roots;
	\item abelian subgroups are finitely generated and undistorted;
	\item they have rational growth.
\end{itemize}
	\item Cactus groups are of type $F_\infty$.
	\item Every subgroup in a cactus group either contains a non-abelian free subgroup or is virtually abelian.
	\item Cactus groups have finite asymptotic dimensions. 
	\item Cactus groups are a-T-menable.
	\item The (equivariant) Hilbert space compression of cactus groups is $1$. 
	\item Cactus groups satisfy the Rapid Decay property.
\end{itemize}
\end{cor}

\noindent
The fact the median graphs on which cactus groups act are explicit allows us to exploit this geometry even more efficiently in order to solve various problems. As an illustration, we improve the second item of the corollary above. 

\begin{thm}
Word and conjugacy problems in cactus groups can be solve efficiently and explicitly.
\end{thm}

\noindent
We refer to Section~\ref{section:WandCP} for precise statements. In particular, we recover the solution to the word problem described in \cite{Paolo} and give it a geometric interpretation. The solution to the conjugacy problem is new.

\paragraph{Conspicial actions.} Since cactus groups naturally act on median graphs, then so do their pure subgroups. It turns out that the induced actions of pure cactus groups thus obtained satisfy remarkable properties.

\begin{thm}
For every $n \geq 2$, the action of the pure cactus group $PJ_n$ on the median graph $\mathscr{C}_n$ is conspicial and has trivial cube-stabilisers.
\end{thm}

\noindent
An action on a median graph is \emph{conspicial}\footnote{This property is called \emph{$C$-special} in \cite{MR2377497}, shortened as \emph{special} in the recent literature. In order to minimise conflicts with the every-day language, we propose to use the Latin word \emph{conspicial} which is close to \emph{special} both semantically and phonetically.} if there is no \emph{self-intersection}, \emph{self-osculation}, \emph{inter-osculation} as defined in Section~\ref{section:Pure}. This additional information allows us to deduce various properties holding for (pure) cactus groups, such as:

\begin{cor}
The following assertions hold. 
\begin{itemize}[itemsep=0pt]
	\item Cactus groups are linear over $\mathbb{Z}$, and in particular residually finite.
	\item Cactus groups have linear residual growth.
	\item Convex cocompact subgroups in pure cactus groups are virtual retracts, hence separable.
	\item Pure cactus groups are residually nilpotent.
	\item Every subgroup in a cactus group either is SQ-universal or contains a non-abelian free subgroup. 
\end{itemize}
\end{cor}

\noindent
A fundamental consequence of admitting a conspicial action on a median graph is that this allows us to embed our group into a right-angled Coxeter group. 

\begin{cor}
Fix an $n \geq 2$ and let $\Gamma$ denote the graph whose vertices are the subsets in $\{1, \ldots, n\}$ and whose edges connect two subsets whenever they are disjoint or nested. The pure cactus group $PJ_n$ naturally embeds into the right-angled Coxeter group $C(\Gamma)$.
\end{cor}

\noindent
We refer to Section~\ref{section:Pure} for an explicit description of the embedding. Interestingly, it seems that this embedding coincides with the embedding constructed in \cite{MR3988817}. 

\medskip \noindent
The fact that pure cactus groups act on median graphs with trivial cube-stabilisers also immediately implies that:

\begin{cor}
Pure cactus groups are torsion-free.
\end{cor}

\noindent
A combinatorial proof of this statement can be found in \cite{Paolo}.

\paragraph{Hyperbolicity.} The fact that cactus groups act properly and cocompactly on median graphs motivates the idea that these groups are nonpositively curved. But is it possible to find an action on negatively curved space? In other words, do cactus groups naturally act on (Gromov-)hyperbolic spaces? Recall that a geodesic metric spaces (e.g.\ a graph) is \emph{hyperbolic} if there exists some $\delta \geq 0$ such that, for every geodesic triangle, any side is contained in the $\delta$-neighbourhood of the union of the two other sides. 

\medskip \noindent
It is well-known in geometric group theory that a lot of valuable information can be extracted from a good action on a hyperbolic space, leading to various notions of hyperbolicity in groups. The strongest notion is given by \emph{hyperbolic groups}, which are finitely generated groups whose Cayley graphs (constructed from finite generating sets) are themselves hyperbolic. However, most cactus groups are not hyperbolic. Same thing for the weaker notions of relative hyperbolicity and cyclic hyperbolicity. See Proposition~\ref{prop:NotHyp} below. Nevertheless, it turns out that most cactus groups are \emph{acylindrically hyperbolic}.

\begin{thm}
For every $n \geq 4$, the cactus group $J_n$ is acylindrically hyperbolic.
\end{thm}

\noindent
Recall that a group $G$ is \emph{acylindrically hyperbolic} if it admits an action on some hyperbolic space $X$ which is \emph{non-elementary} (i.e.\ with an infinite limit set) and \emph{acylindrical} (i.e.\ for every $D \geq 0$, there exist $L,N \geq 0$ such that $\# \{ g \in G \mid d(x,gx),d(y,gy) \leq D\} \leq N$ for all $x,y \in X$ satisfying $d(x,y) \geq L$). We refer to the survey \cite{MR3966794} for more information on acylindrically hyperbolic groups.

\medskip \noindent
Knowing that a group is acylindrically hyperbolic automatically implies various properties, including:

\begin{cor}
Fix an $n \geq 4$ and let $G$ be either the cactus group $J_n$ or the pure cactus group $PJ_n$. The following assertions hold.
\begin{itemize}[itemsep=0pt]
	\item $G$ is \emph{SQ-universal} (i.e.\ every countable group embeds into a quotient of $G$).
	\item $G$ contains uncountably many normal free subgroups.
	\item $G$ satisfy the \emph{Property $P_{\text{naive}}$} (i.e.\ for all $h_1, \ldots, h_r \in G$, there exists $g \in G$ such that $\langle h_i,g \rangle = \langle h_i \rangle \ast \langle g \rangle$ for every $1 \leq i \leq r$).
	\item Random subgroups in $G$ are free (see \cite{MR4023758} for a precise statement).
	\item The space of quasimorphisms of $G$ is infinite-dimensional. As a consequence, the the second bounded cohomology group of $G$ is infinite-dimensional.
	\item $G$ does not contain any infinite abelian $s$-normal subgroup. (Recall that an subgroup $H$ is $s$-normal if $H \cap gHg^{-1}$ is infinite for every $g \in G$.)
	\item $G$ does not decompose as a product of two infinite groups. 
	\item Asymptotic cones of $G$ contain cut points. 
\end{itemize}
\end{cor}

\noindent
The sixth item implies that most (pure) cactus groups have finite centres. In fact, from there it is not difficult to upgrade this observation as:

\begin{cor}
For all $n \geq 3$ and $m \geq 4$, $PJ_m$ and $J_n$ have trivial centres.
\end{cor}

\noindent
A combinatorial proof of this statement is also available in \cite{Paolo}. In fact, we are able to show more generally that every finite normal subgroup in a cactus group is necessarily trivial. See Corollary~\ref{cor:NoFiniteNormal}. 

\paragraph{Acknowledgements.} I am grateful to T. Haettel for his comments on a preliminary version of the present article.

\section{Median geometry}\label{section:MedianGeometry}

\noindent
In this section, we prove that the canonical Cayley graphs of cactus groups are median. We start by recalling basic definitions and results related to median graphs in Section~\ref{section:Crash}. In Section~\ref{section:Normal} we introduce a normal on interval diagrams, which we use in Section~\ref{section:CactusMedian} in order to prove our main result.

\subsection{A crash course on median graphs}\label{section:Crash}

\noindent
A connected graph $X$ is \emph{median} if, for all vertices $x_1,x_2,x_3 \in X$, there exists a unique vertex $m \in X$ satisfying
$$d(x_i,x_j)=d(x_i,m)+d(m,x_j) \text{ for all } i \neq j.$$
We refer to this vertex $m$ as the \emph{median point} of $x_1,x_2,x_3$. 
\begin{figure}[h!]
\begin{center}
\includegraphics[scale=0.4]{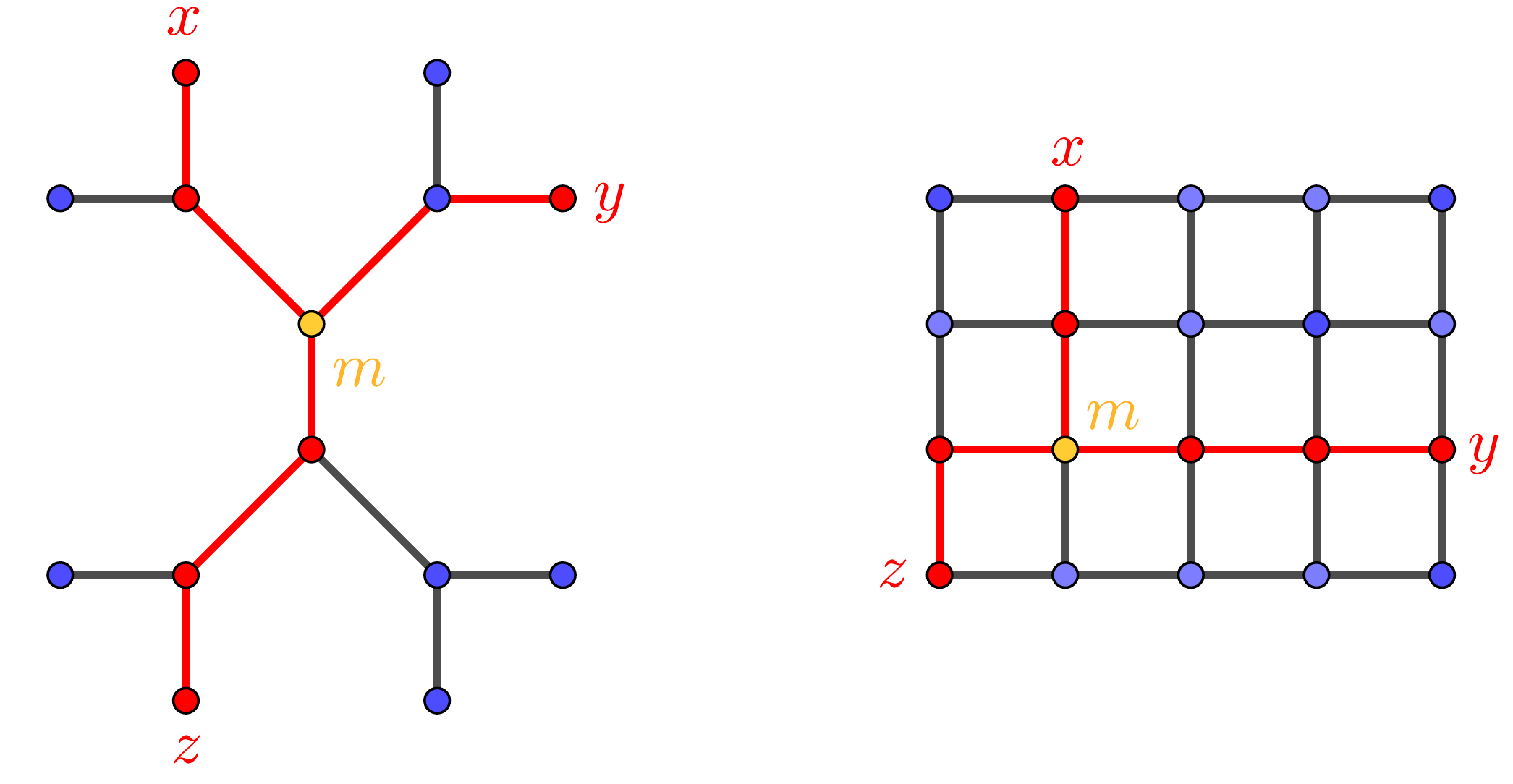}
\caption{Examples of median graphs and median points.}
\label{ExMedian}
\end{center}
\end{figure}

\medskip \noindent
Examples include of course simplicial trees, the median point of a triple of vertices corresponding to the centre of the tripod they delimit. A product of median graphs is still median, so product of trees are also median graphs. This includes in particular (one-skeleta of) cubes of arbitrary dimensions. 

\medskip \noindent
A fundamental idea is that the geometry of a median graph is essentially encoded in the combinatorics of its \emph{hyperplanes}. 

\begin{definition}
Let $X$ be a median graph. A(n \emph{oriented}) \emph{hyperplane} $J$ is an equivalence class of (oriented) edges with respect to the transitive closure of the relation that identifies two (oriented) edges when they are opposite sides of a $4$-cycle. 
\begin{itemize}
	\item If $X \backslash \backslash J$ denotes the graph obtained from $X$ by removing all the edges of $J$, then a connected component of $X \backslash \backslash J$ is a \emph{halfspace}. Two subsets $A,B \subset X$ are \emph{separated} by $J$ if they lie in distinct halfspaces delimited by $J$.
	\item The subgraph $N(J)$ spanned by all the edges of a hyperplane $J$ is its \emph{carrier}. The connected components of $N(J) \backslash \backslash J$ are the \emph{fibres} of the hyperplane. 
	\item Two hyperplanes $J_1$ and $J_2$ are \emph{transverse} if there exist two intersecting edges $e_1 \subset J_1$ and $e_2 \subset J_2$ that span a $4$-cycle. 
	\item They are \emph{tangent} if exist two intersecting edges $e_1 \subset J_1$ and $e_2 \subset J_2$ that do not span a $4$-cycle.
\end{itemize}
See Figure~\ref{Hyp} for a few examples.
\end{definition}
\begin{figure}[h!]
\begin{center}
\includegraphics[scale=0.4]{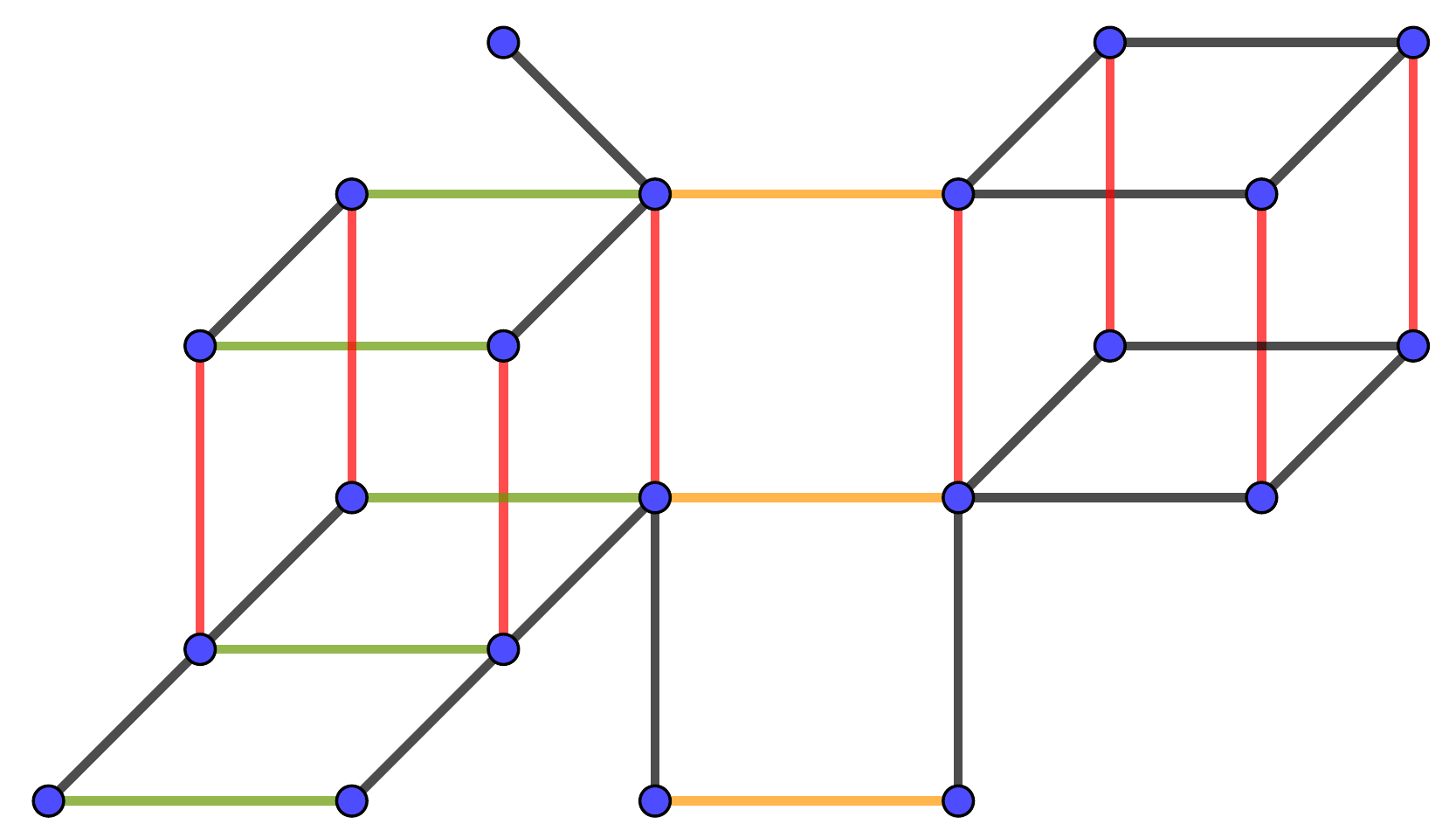}
\caption{Hyperplanes in a median graph. The red hyperplane is transverse to the green and yellow hyperplanes. Green and yellow hyperplanes are tangent.}
\label{Hyp}
\end{center}
\end{figure}

\noindent
The previous claim is mainly motivated by the following statement:

\begin{thm}[\cite{MR1347406}]\label{thm:MedianBig}
Let $X$ be a median graph. The following assertions hold.
\begin{itemize}
	\item Every hyperplane $J$ separates. More precisely, $X \backslash \backslash J$ has exactly two connected components.
	\item Halfspaces, carriers, and fibres are convex.
	\item A path in $X$ is a geodesic if and only if it crosses each hyperplane at most once.
	\item The distance between two vertices $x,y \in X$ coincides with the number of hyperplanes separating $x,y$. 
\end{itemize}
\end{thm}

\noindent
In geometric group theory, median graphs are better known as \emph{CAT(0) cube complexes}, but it turns out that CAT(0) cube complexes and median graphs define essentially the same objects. More precisely:

\begin{thm}[\cite{MR1663779, mediangraphs, Roller}]\label{thm:MedianVScube}
A graph is median if and only if its cube-completion is CAT(0). 
\end{thm}

\noindent
Here, given a graph, its \emph{cube-completion} refers to the cube complex obtained by filling with cubes all the subgraphs isomorphic to one-skeleta of cubes. By abuse of language, we refer to a \emph{cube} in a graph as an induced subgraph isomorphic to a product of edges.

\medskip \noindent
In order to prove that Cayley graphs of cactus groups are median, we will use the following criterion:

\begin{thm}\label{thm:CriterionCC}
Let $X$ be a cube complex. Assume that the following conditions hold:
\begin{itemize}
	\item[(i)] $X$ is simply connected;
	\item[(ii)] the squares in $X$ are embedded;
	\item[(iii)] two distinct squares never share two consecutive edges;
	\item[(iv)] a cycle of three squares spans the two-skeleton of a $3$-cube.
\end{itemize}
Then the one-skeleton $X^{(1)}$ of $X$ is a median graph. Moreover, every $4$-cycle in $X^{(1)}$ bounds a square in $X$. 
\end{thm}

\noindent
The fourth item  of the theorem is illustrated by Figure~\ref{Condition}.
\begin{figure}[h!]
\begin{center}
\includegraphics[width=0.6\linewidth]{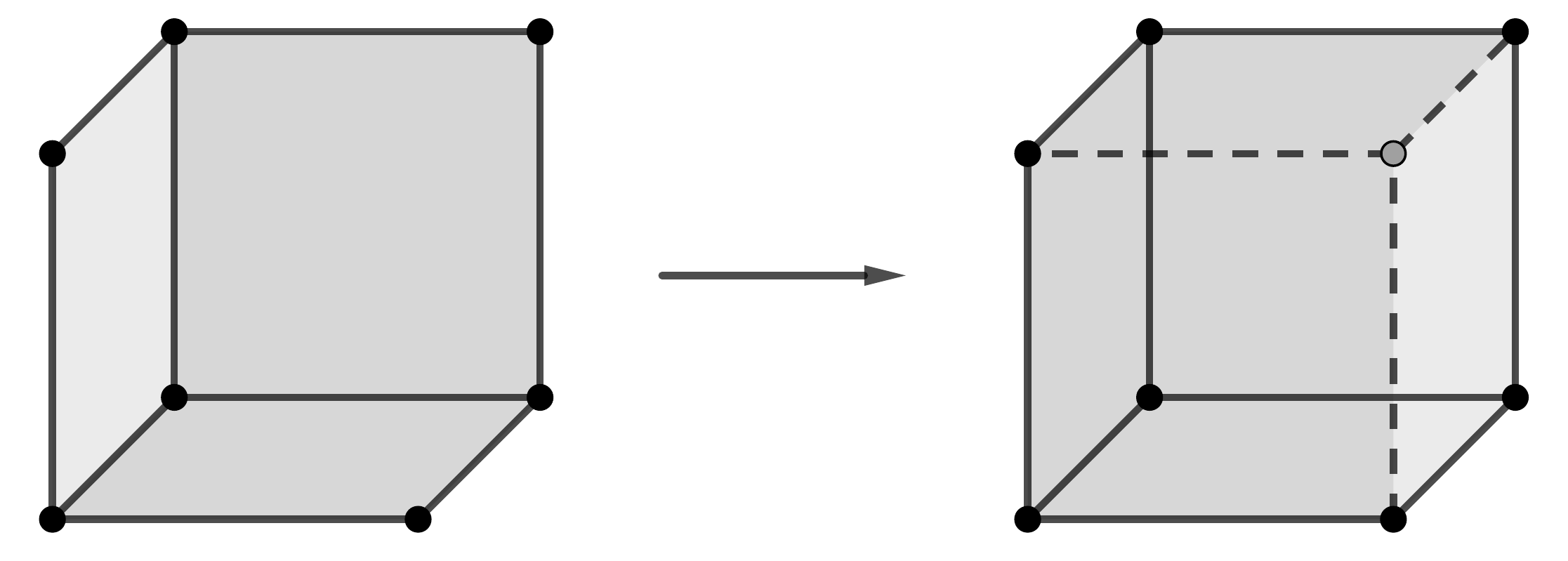}
\caption{A cycle of three squares spans the two-skeleton of a $3$-cube.}
\label{Condition}
\end{center}
\end{figure}

\begin{proof}[Sketch of proof of Theorem~\ref{thm:CriterionCC}.]
Because $X$ is simply connected, every loop delimits a \emph{disc diagram}. As a consequence of our assumptions, hexagonal moves and reductions can be applied to disc diagrams, so that the analysis of disc diagrams of minimal area from \cite[Section~2.e]{MR4298722} still holds, which allows us to recover most of the main properties of hyperplanes \cite[Sections~2.f and~2.g]{MR4298722}. 

\medskip \noindent
Let $x,y,z \in X$ be three vertices. According to \cite[Corollary~2.15]{MR4298722}, halfspaces are convex, so a median point, if it exists, has to belong to the intersection $I$ of all the halfspaces containing at least two vertices among $x,y,z$. Because any two vertices are separated by a hyperplane and because each hyperplane delimits exactly two halfspaces \cite[Corollaries~2.15 and~2.16]{MR4298722}, we know that $I$ is either empty or reduced to a single vertex. In order to conclude the proof of our theorem, it suffices to show that $I$ is non-empty.

\medskip \noindent
Distinguishing the halfspaces containing the three vertices $x,y,z$ from the other halfspaces, we write $I$ and $A \cap B$, where $A$ is the intersection of all the halfspaces containing $x,y,z$ and where $B$ is the intersection of all the halfspaces containing exactly two vertices among $x,y,z$. According to \cite[Lemma~2.19]{MR4298722}, $A$ coincides with the convex hull of $\{x,y,z\}$. Because only finitely many hyperplanes separate two given vertices \cite[Corollary~2.16]{MR4298722}, $B$ is an intersection of finitely many pairwise intersecting halfspaces. By applying the Helly property \cite[Lemma~2.10]{MR4298722}, we conclude that $I=A \cap B$ is non-empty, as desired. 

\medskip \noindent
Thus, we have proved the first assertion of our theorem, namely the one-skeleton of $X$ is a median graph. Next, given a $4$-cycle in the one-skeleton of $X$, consider a disc diagram of minimal area it delimits. It follows from \cite[Corollary~2.4]{MR4298722} that such a disc diagram must be reduced to a single square, proving the second assertion of our theorem. 
\end{proof}

\noindent
For more information on median graphs, or equivalently CAT(0) cube complexes, we refer to \cite{MR3329724, Book}.

\subsection{A normal form}\label{section:Normal}

\noindent
As mentioned in the introduction, any element of the cactus group $J_n$ can be represented by a braid-like picture with $n$ strands. In the article, we will often use a slightly different, but equivalent, representation of the elements of $J_n$. A \emph{diagram with $n$ strands} is a collection $n$ parallel segments drawn vertically on the plane (the \emph{strands}) with finitely many pairwise disjoint horizontal segments connecting the strands (the \emph{intervals}). Two diagrams are considered as equal if one is the image of the other under an isotopy of the plane.  

\medskip \noindent
\begin{minipage}{0.18\linewidth}
\begin{center}
\includegraphics[scale=0.48]{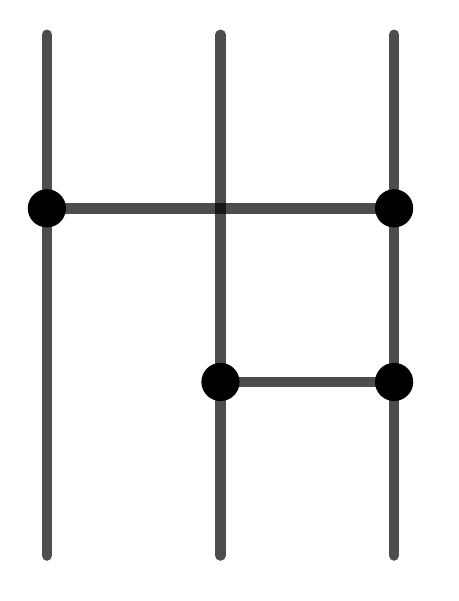}
\end{center}
\end{minipage}
\begin{minipage}{0.80\linewidth}
An element of $J_n$, say $s_{i_1,j_1} \cdots s_{i_r,j_r}$, is represented by the diagram with $n$ strands obtained by adding from top to bottom intervals connecting the strands $i_k$ and $j_k$, $k=1, \ldots, r$. For instance, the element $s_{1,3}s_{2,3}$ of $J_3$ is represented by the diagram given on the left. Conversely, every diagram with $n$ strands represents an element of $J_n$. One recovers the braided-like picture by ``pinching'' the horizontal intervals.
\end{minipage}

\medskip \noindent
The \emph{length} of a diagram is defined as its number of intervals. The relations of the group can be translated as simple relations between diagrams.

\medskip \noindent
\includegraphics[width=\linewidth]{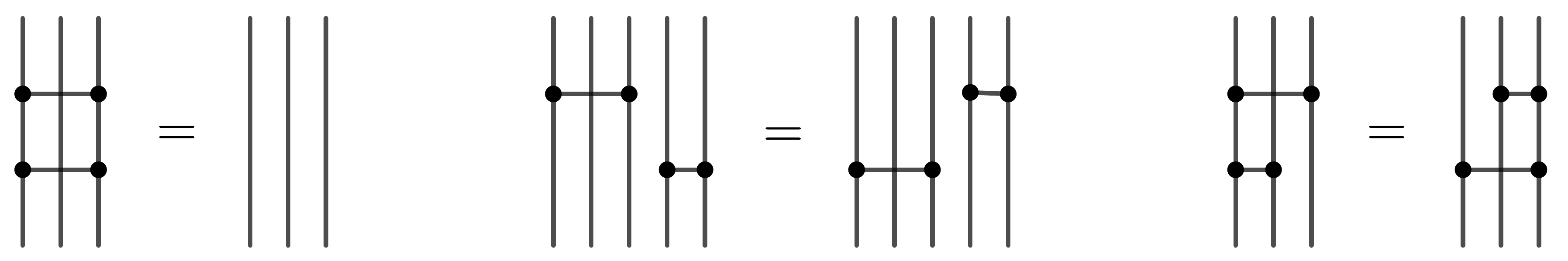}

\medskip \noindent
The first relation is referred to as a \emph{reduction} and the third as a \emph{flip}. The second relation justifies why diagrams containing intervals lying on the same horizontal line represent elements of the cactus group, even though they do not represent unique words of generators. A diagram is \emph{reduced} if no reduction applies, even after applying a sequence of flips. 

\medskip \noindent
In order to avoid confusion, in the sequel we denote by $=$ the equality between diagrams and by $\equiv$ the equality modulo flips and reductions. Observe that two diagrams are $\equiv$-equal if and only if they represent the same element in the cactus group. In the rest of the section, we show that it possible to define canonical representatives of $\equiv$-equivalent classes. 

\begin{definition}
A diagram is \emph{normal} if it is reduced and if no flip can applied in order to raise an interval above a smaller interval. 
\end{definition}

\noindent
Clearly, every diagram can be put into a normal form by reducing and flipping. It turns out that this normal form is unique.

\begin{prop}\label{prop:Normal}
For every diagram $\Delta$, there exists a unique normal diagram $\Delta_0$ such that $\Delta \equiv \Delta_0$. Moreover, $\Delta_0$ can be obtained from $\Delta$ by a sequence reductions and flippings of pairs of intervals that raise the bigger interval. 
\end{prop}

\begin{proof}
Let $\mathcal{G}(\Delta)$ denote the oriented graph whose vertices are all the diagrams $\equiv$-equal to $\Delta$ and whose oriented edges connect a diagram $\Delta_1$ to a diagram $\Delta_2$ if $\Delta_2$ can be obtained from $\Delta_1$ by a reduction or by flipping a pair of intervals in order to raise the bigger interval. For short, we write $\Delta_1 \to \Delta_2$ if there is an oriented edge from $\Delta_1$ to $\Delta_2$. We also write $\Delta_1 \overset{\ast}{\to} \Delta_2$ if there exists an oriented path from $\Delta_1$ to $\Delta_2$. 

\medskip \noindent
First, observe that (the underlying unoriented graph of) $\mathcal{G}(\Delta)$ is connected. This amounts to saying that any two $\equiv$-equal diagrams can be related by a sequence of reductions, expansions, and flips. This is true by the very definition of $\equiv$.

\medskip \noindent
Next, observe that $\mathcal{G}(\Delta)$ is \emph{terminating}, i.e.\ there is no infinite oriented ray. It suffices to introduce a \emph{complexity} $\chi(\cdot) \in \mathbb{N}$ satisfying $\chi(\Phi) < \chi(\Psi)$ for all diagrams $\Psi \to \Phi$. Let $\Omega$ be a diagram. Thinking of $\Omega$ as braided-like picture, $\chi(\Omega)$ is the number of pairs of nodes $(n,m)$ such that all the strands passing through $n$ pass through $m$ as well and such that $n$ lies above $m$. One easily sees that, if $\Phi,\Psi$ are two diagrams satisfying $\Psi \to \Phi$, then $\chi(\Phi)< \chi(\Psi)$. This proves our claim.

\medskip \noindent
Finally, observe that $\mathcal{G}(\Delta)$ is \emph{locally confluent}, i.e.\ for all distinct diagrams $\Omega,\Phi,\Psi$ satisfying $\Omega \to \Phi$ and $\Omega \to \Psi$, there exists a diagram $\Xi$ such that $\Phi,\Psi \overset{\ast}{\to} \Xi$. There are several cases to consider depending on whether $\Phi,\Psi$ are obtained from $\Omega$ by a reduction or a flip. 

\medskip \noindent
\emph{Case 1:} $\Phi,\Psi$ are both obtained from $\Omega$ by reduction. If the two pairs of intervals corresponding to the reductions $\Omega \to \Phi$ and $\Omega \to \Psi$ intersect, then we must have $\Phi=\Psi$, contrary to our assumptions. If the pairs are disjoint, then performing the two reductions simultaneously provides a diagram $\Xi$ satisfying $\Phi,\Psi \to \Xi$.

\medskip \noindent
\emph{Case 2:} $\Phi$ is obtained by a reduction and $\Psi$ by a flip. If the two pairs of intervals corresponding to the operations $\Omega \to \Phi$ and $\Omega \to \Psi$ are disjoint, then again performing the two operations simultaneously provides a diagram $\Xi$ satisfying $\Phi,\Psi \to \Xi$. If the two pairs intersect, then, as justified by Figure~\ref{ConfluentOne}, $\Xi:= \Phi$ is the diagram we are looking for. 
\begin{figure}[h!]
\begin{center}
\includegraphics[width=0.7\linewidth]{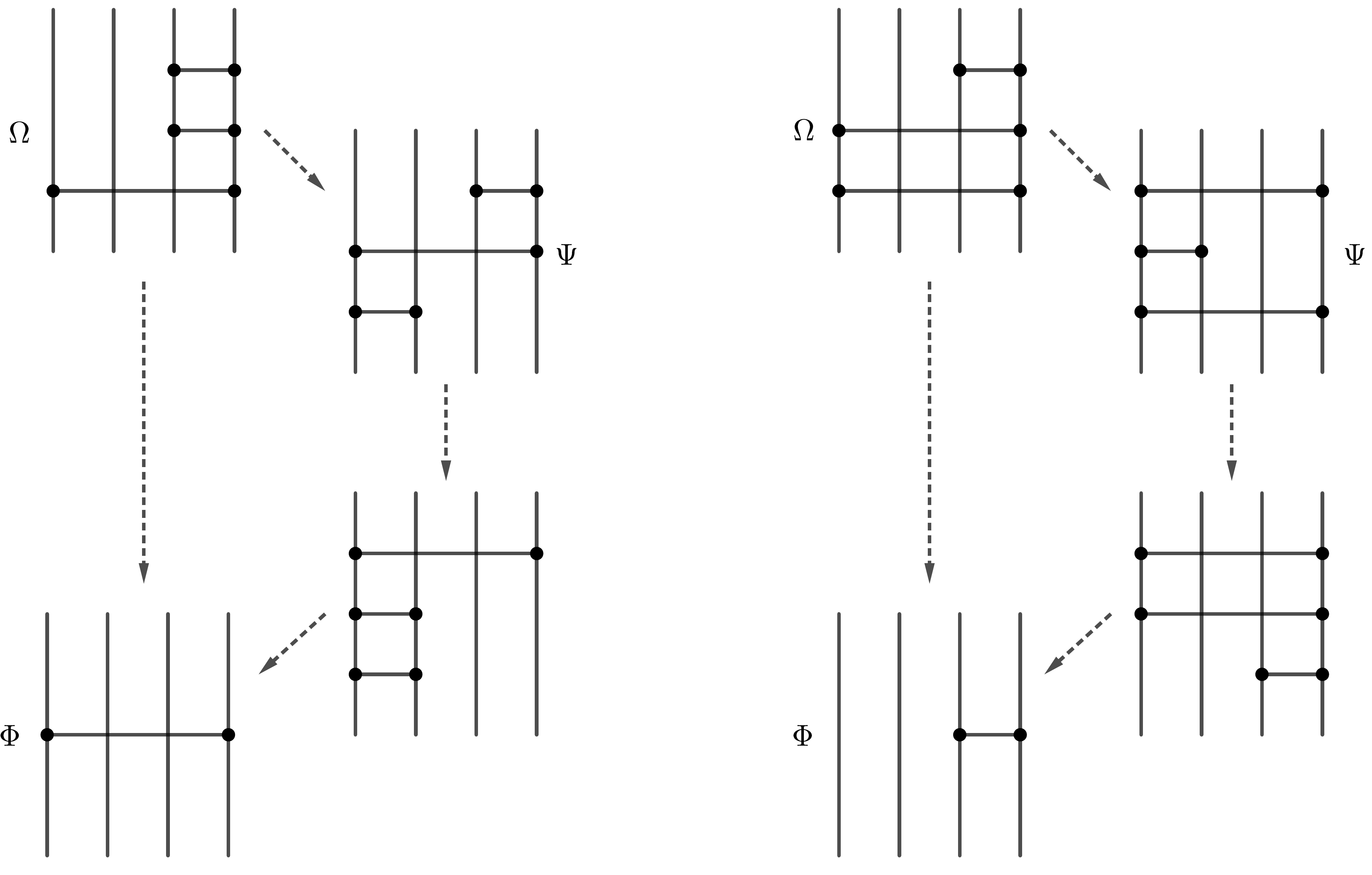}
\caption{Configuration from Case 2 in the proof of Proposition~\ref{prop:Normal}.}
\label{ConfluentOne}
\end{center}
\end{figure}

\medskip \noindent
\emph{Case 3:} $\Phi$ is obtained by a flip and $\Psi$ by a reduction. Exchanging the role played by $\Phi$ and $\Psi$, we are in the same situation as Case 2.

\medskip \noindent
\emph{Case 4:} $\Phi,\Psi$ are both obtained from $\Omega$ by a flip. If the two pairs of intervals corresponding to the flips $\Omega \to \Phi$ and $\Omega \to \Psi$ are disjoint, then performing the two flips simultaneously provides a diagram $\Xi$ satisfying $\Phi,\Psi \to \Xi$. Otherwise, the diagram $\Xi$ is constructed as illustrated by Figure~\ref{ConfluentTwo}.
\begin{figure}[h!]
\begin{center}
\includegraphics[width=0.8\linewidth]{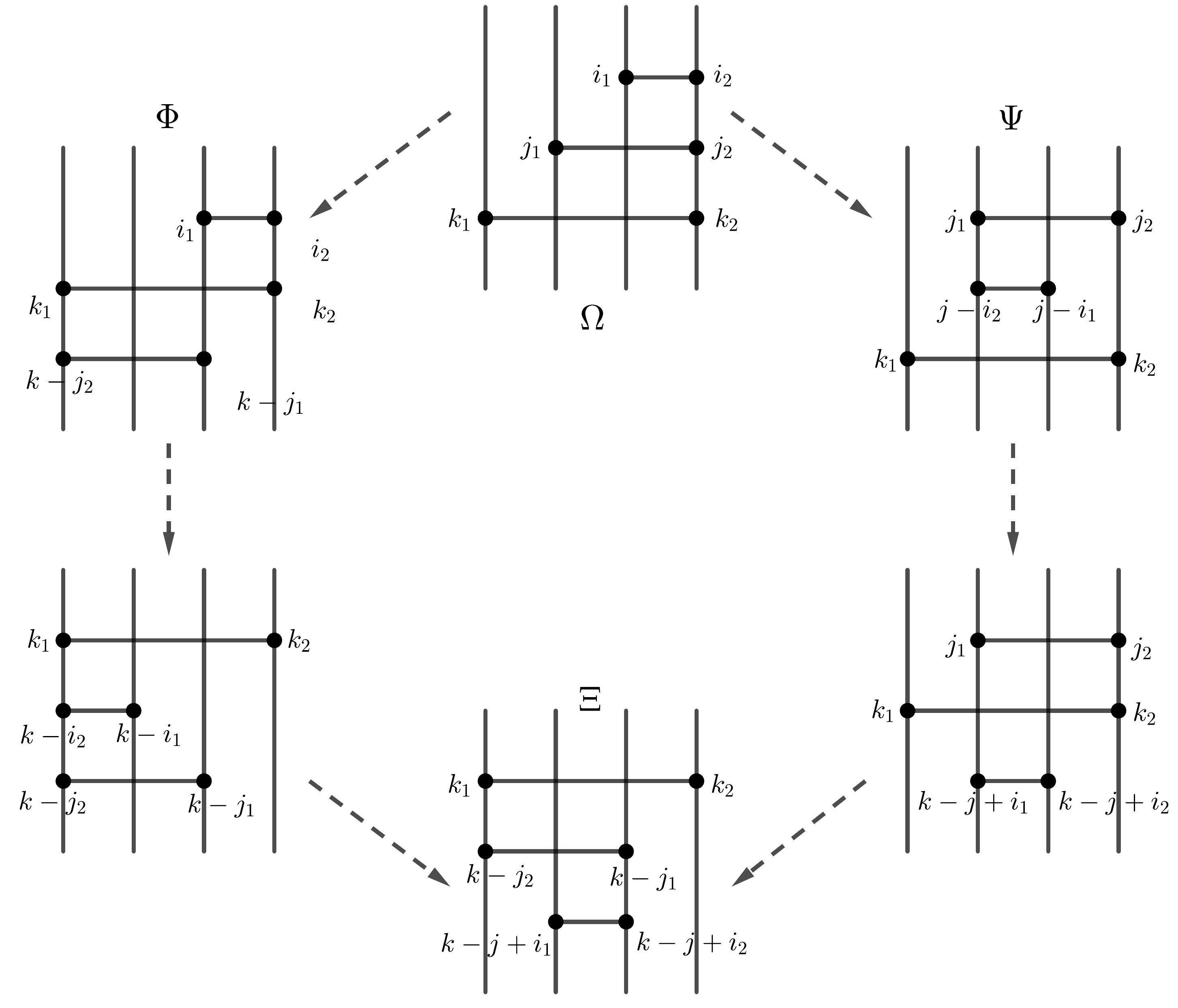}
\caption{Case 4 in the proof of Proposition~\ref{prop:Normal}, where $j:=j_1+j_2$ and $k:=k_1+k_2$.}
\label{ConfluentTwo}
\end{center}
\end{figure}

\medskip \noindent
One easily shows that $\mathcal{G}(\Delta)$ locally confluent implies that $\mathcal{G}(\Delta)$ is \emph{confluent}, i.e.\ for all distinct diagrams $\Omega,\Phi,\Psi$ satisfying $\Omega \overset{\ast}{\to} \Phi$ and $\Omega \overset{\ast}{\to} \Psi$, there exists a diagram $\Xi$ such that $\Phi,\Psi \overset{\ast}{\to} \Xi$; see for instance \cite[Theorem~1]{MR7372} or \cite[Section~4.1]{MR1127191}. Our proposition now follows by noticing that a vertex in $\mathcal{G}(\Delta)$ corresponds to a normal diagram if there is no oriented edge starting from it. Indeed, the fact that $\mathcal{G}(\Delta)$ is terminating implies that every diagram can be turned into a normal diagram by reducing and flipping pairs of interval that raise the bigger interval; and because $\mathcal{G}(\Delta)$ is confluent, two normal diagrams obtained from two diagrams $\equiv$-equal to $\Delta$ must coincide. 
\end{proof}

\subsection{Median geometry of cactus groups}\label{section:CactusMedian}

\noindent
In this section, we use the normal form of diagrams described in the previous section in order to show that the canonical Cayley graphs of cactus groups are median graphs. Consequently, cactus groups act properly, cocompactly, and vertex-freely on median graphs.

\begin{thm}\label{thm:CAT}
Let $n \geq 2$. The Cayley graph $\mathscr{C}_n$ of the cactus group $J_n$ with respect to its canonical generators is a median graph. 
\end{thm}

\begin{proof}
Let $\mathscr{SC}_n$ denote the Cayley complex of the presentation of $J_n$, i.e.\ the square complex whose one-skeleton is $\mathscr{C}_n$ and whose squares are bounded by the following $4$-cycles:
\begin{itemize}
	\item $(1, s_{p,q}, s_{p,q} s_{m,r}, s_{p,q} s_{m,r} s_{p,q}, s_{p,q}s_{m,r}s_{p,q} s_{m,r}=1)$ for all $1 \leq p < q \leq n$ and $1 \leq m < r \leq n$ satisfying $[p,q] \cap [m,r]= \emptyset$;
	\item $(1, s_{p,q}, s_{p,q}s_{m,r}, s_{p,q}s_{m,r}s_{p,q}, s_{p,q}s_{m,r} s_{p,q} s_{p+q-r,p+q-m}=1)$ for all $1 \leq p < q \leq n$ and $1 \leq m < r \leq n$ satisfying $[m,r] \subset [p,q]$. 
\end{itemize}
In other words, the squares are given by the relations from the canonical presentation of $J_n$. This implies that $\mathscr{SC}_n$ is simply connected. Because a word of generators given by a relation from the presentation of $J_n$ is uniquely determined by any two consecutive letters, we know that two squares in $\mathscr{SC}_n$ cannot share two consecutive sides. Moreover, we deduce from Proposition~\ref{prop:Normal} that the squares in $\mathscr{SC}_n$ are embedded. 

\medskip \noindent
Thus, in order to deduce from Theorem~\ref{thm:CriterionCC} that $\mathscr{C}_n$ is a median graph, it suffices to verify that, given a vertex $o \in \mathscr{C}_n$ and three neighbours $x_1,x_2,x_3 \in \mathscr{C}_n$, if the edges $[o,x_1],[o,x_2],[o,x_3]$ pairwise span a square in $\mathscr{SC}_n$ then they globally span the two-skeleton of a $3$-cube. Because $J_n$ acts on $\mathscr{C}_n$ vertex-transitively, we can suppose without loss of generality that $o=1$. For $i=1,2,3$, let $S_i$ denote the set of strands braided by $x_i$. Notice that, for all $1 \leq i \neq j \leq 3$, because $[o,x_i]$ and $[o,x_j]$ span a square, $S_i$ and $S_j$ must be either disjoint or nested. For every $0 \leq r \leq 3$ and for all $1 \leq i_1< \ldots < i_r \leq 3$, let $x(i_1, \ldots, i_r)$ denote the braid-like picture obtained by inverting the strands in $S_{i_1}$, next the strands in $S_{i_2}$, and so on; notice that $x(i_1, \ldots, i_r)$ is well-defined because the $S_i$ are pairwise disjoint or nested. Observe that $x(i_1,\ldots, i_r)$ coincides with $o=1$ if $r=0$ and with $x_i$ if $r=1$ and $i_1=i$. The vertices $x(i_1, \ldots, i_r)$ are the vertices are the $3$-cube we are looking for. See Figure~\ref{Cube} for the illustration of a particular case. 
\end{proof}
\begin{figure}[h!]
\begin{center}
\includegraphics[width=0.6\linewidth]{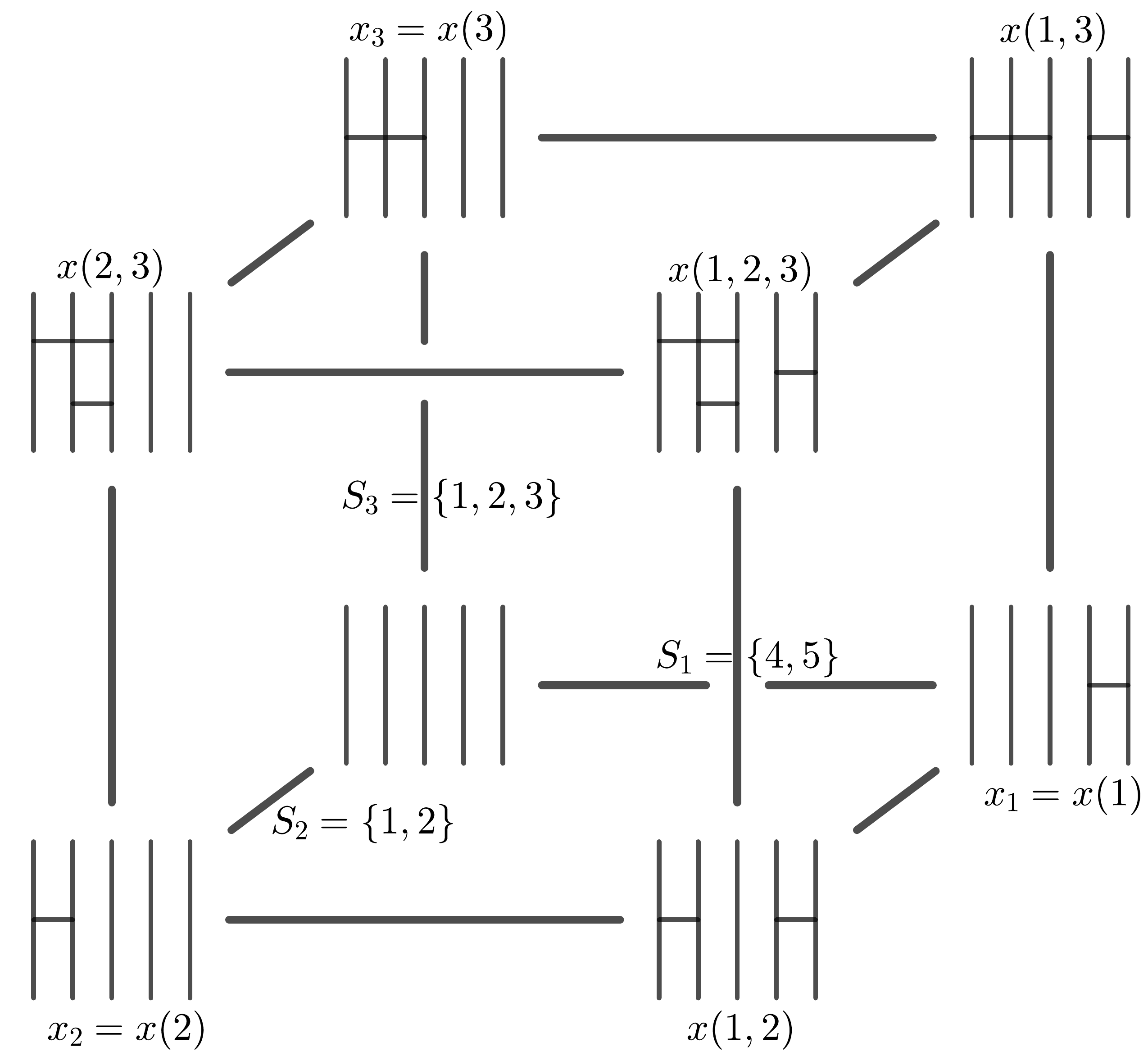}
\caption{Cube constructed in the proof of Theorem~\ref{thm:CAT}.}
\label{Cube}
\end{center}
\end{figure}

\noindent
As a by-product of the previous proof, combined with the description of $4$-cycles from Theorem~\ref{thm:CriterionCC}, we deduce that:

\begin{cor}\label{cor:Cycles}
Let $n \geq 2$. All the $4$-cycles of $\mathscr{C}_n$ correspond to mock commutations, i.e.\ they are of one of the following forms:
\begin{itemize}
	\item $(g, gs_{p,q}, gs_{p,q} s_{m,r}, gs_{p,q} s_{m,r} s_{p,q}, gs_{p,q}s_{m,r}s_{p,q} s_{m,r}=g)$ for some $1 \leq p < q \leq n$ and $1 \leq m < r \leq n$ satisfying $[p,q] \cap [m,r]= \emptyset$ and some $g \in J_n$;
	\item $(g, gs_{p,q}, gs_{p,q}s_{m,r}, gs_{p,q}s_{m,r}s_{p,q}, gs_{p,q}s_{m,r} s_{p,q} s_{p+q-r,p+q-m}=g)$ for some $1 \leq p < q \leq n$ and $1 \leq m < r \leq n$ satisfying $[m,r] \subset [p,q]$ and some $g \in J_n$. 
\end{itemize}
\end{cor}

\section{Word and conjugacy problems}\label{section:WandCP}

\noindent
It is known that CAT(0) groups, including groups acting properly and cocompactly on median graphs, have solvable word and conjugacy problems (see for instance \cite{MR1744486}). Therefore, Theorem~\ref{thm:CAT} immediately implies that word and conjugacy problems are solvable in cactus groups. However, the algorithms thus obtained are not applicable by hand, which is not entirely satisfying. In this section, we show that the median structure of Cayley graphs of cactus groups allows us to solve the word and conjugacy problems very explicitly and efficiently.

\subsection{Word problem}\label{section:Word}

\noindent
An efficient solution to the word problem in a cactus group is already given by Proposition~\ref{prop:Normal}. Given two words of generators $w_1$ and $w_2$, draw the interval diagrams of $w_1w_2^{-1}$, and find its normal form by raising each interval as much as possible. The words $w_1$ and $w_2$ represent the same element in the cactus group if and only if the normal diagram contains no interval. Observe that the number of required operations is about the square of $|w_1|+|w_2|$ (where $|\cdot|$ denotes the length of the word under consideration).

\medskip \noindent
Let us describe another solution to the word problem in $J_n$, given in \cite{Paolo}, because it admits a nice geometric interpretation. 

\medskip \noindent
Let $X$ be a median graph and $\gamma \subset X$ a path. If $\gamma$ is not a geodesic, then there exists a hyperplane crossing $\gamma$ twice (Theorem~\ref{thm:MedianBig}). Let $e_1,e_2 \subset \gamma$ be two edges of $\gamma$ crossed by the same hyperplane, say $H$. We choose $e_1,e_2$ as close as possible. Consequently, the subsegment $\alpha$ of $\gamma$ between $e_1,e_2$ cannot cross a hyperplane twice, so it must be a geodesic. Since fibres of hyperplanes are convex (Theorem~\ref{thm:MedianBig}), $\alpha$ lies in a fibre of $H$.

\medskip \noindent
\begin{minipage}{0.55\linewidth}
\includegraphics[width=0.95\linewidth]{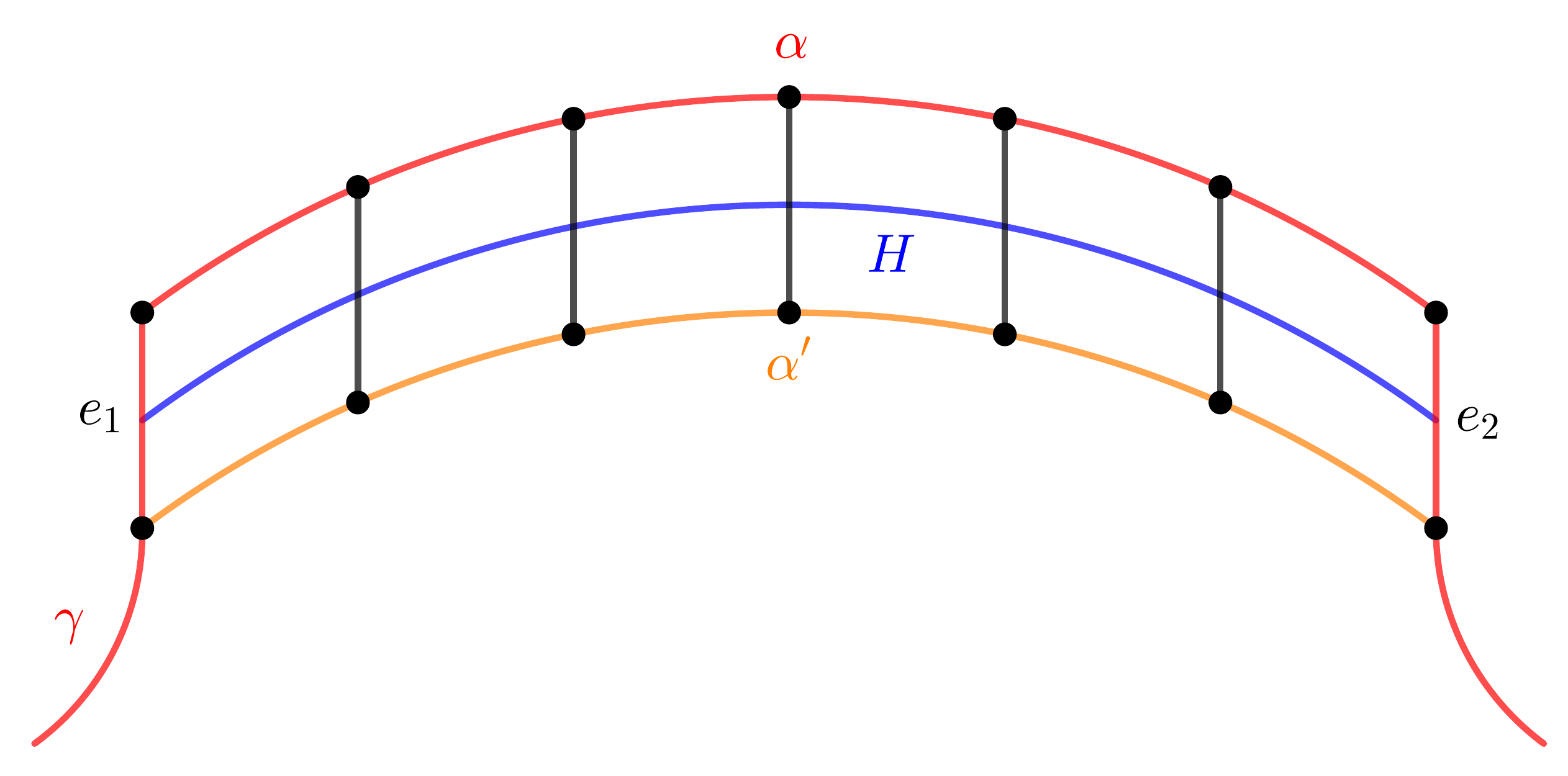}
\end{minipage}
\begin{minipage}{0.43\linewidth}
Replacing $e_1 \cup \alpha \cup e_2$ by the image $\alpha'$ of $\alpha$ in the opposite fibre of $H$ shortens $\gamma$. By iterating the process, one obtains a geodesic after at most $\mathrm{length}(\gamma)/2$ steps. 
\end{minipage}

\medskip \noindent
Transferring this geometric picture to $J_n$, it follows that every word of generators can be turned into a word of minimal length by applying the following operation as much as possible:
\begin{itemize}
	\item[] if a word $w$ contains two letters $\ell$ and $\ell'$ such that $\ell'$ can be shifted all the way to $\ell$ and in the process becomes $\ell$, then do this and remove the subword $\ell \ell$ thus obtained. 
\end{itemize}
A word to which such an operation cannot be applied is \emph{irreducible}. Our geometric argument above, when applied to the Cayley graph of a cactus group, shows that a word of generators has minimal length if and only if it is irreducible. This method also allows us to solve the word problem efficiently and by hand for short words.

\medskip \noindent
It is worth noticing that we can also justify geometrically that two irreducible words present the same element in the cactus group if and only if one can be obtained from the other by a sequence of mock commutations. This is due to the fact that, in a median graph, given two geodesics with the same endpoints, one can always be obtained from the other by \emph{flipping squares}, i.e.\ replacing two consecutive edges from a square with the other pair of consecutive edges \cite[Theorem~4.6]{MR1347406}.

\subsection{Conjugacy problem}\label{section:Conjugacy}

\noindent
Before describing how to solve the conjugacy problem in cactus groups, we need to recall some basic definitions and results related to diagrams over group presentations. The reference we essentially follow is \cite{LS}. When reading the definitions below, we refer to Figure \ref{diag} for examples; there, $\partial \Delta_2$ is labelled by $(a^{-1}a^{2}a^{-1}, a^3a^{-3})$, and $\partial \Delta_3$ by $b^{-1}a^{-1}ba^4b^{-1}a^{-2}ba^{-2}$. 

\begin{definition}
Let $\mathcal{P}= \langle X \mid R \rangle$ be a group presentation, and $D$ a finite $2$-complex embedded into the plane whose edges are oriented and labelled by elements of $X$. An oriented path $\gamma$ in the one-skeleton of $D$ can be written as a concatenation $e_1^{\epsilon_1} \cdots e_n^{\epsilon_n}$, where $\epsilon_1, \ldots, \epsilon_n \in \{+1,-1\}$ and where each $e_i$ is an edge endowed with the orientation coming from $D$. Then the word \emph{labelling} $\gamma$ is $\ell_1^{\epsilon_1} \cdots \ell_n^{\epsilon_n}$ where $\ell_i$ is the label of $e_i$ for every $1 \leq i \leq n$. If, for every $2$-cell $F$ of $D$, the word labelling the boundary of $F$ (an arbitrary basepoint and an arbitrary orientation being fixed) is a cyclic permutation of a relation of $R$ or of the inverse of a relation of $R$, then $D$ is a \emph{diagram over $\mathcal{P}$}.
\end{definition}
\begin{figure}
\begin{center}
\includegraphics[trim={0 6.5cm 10cm 0},clip,scale=0.34]{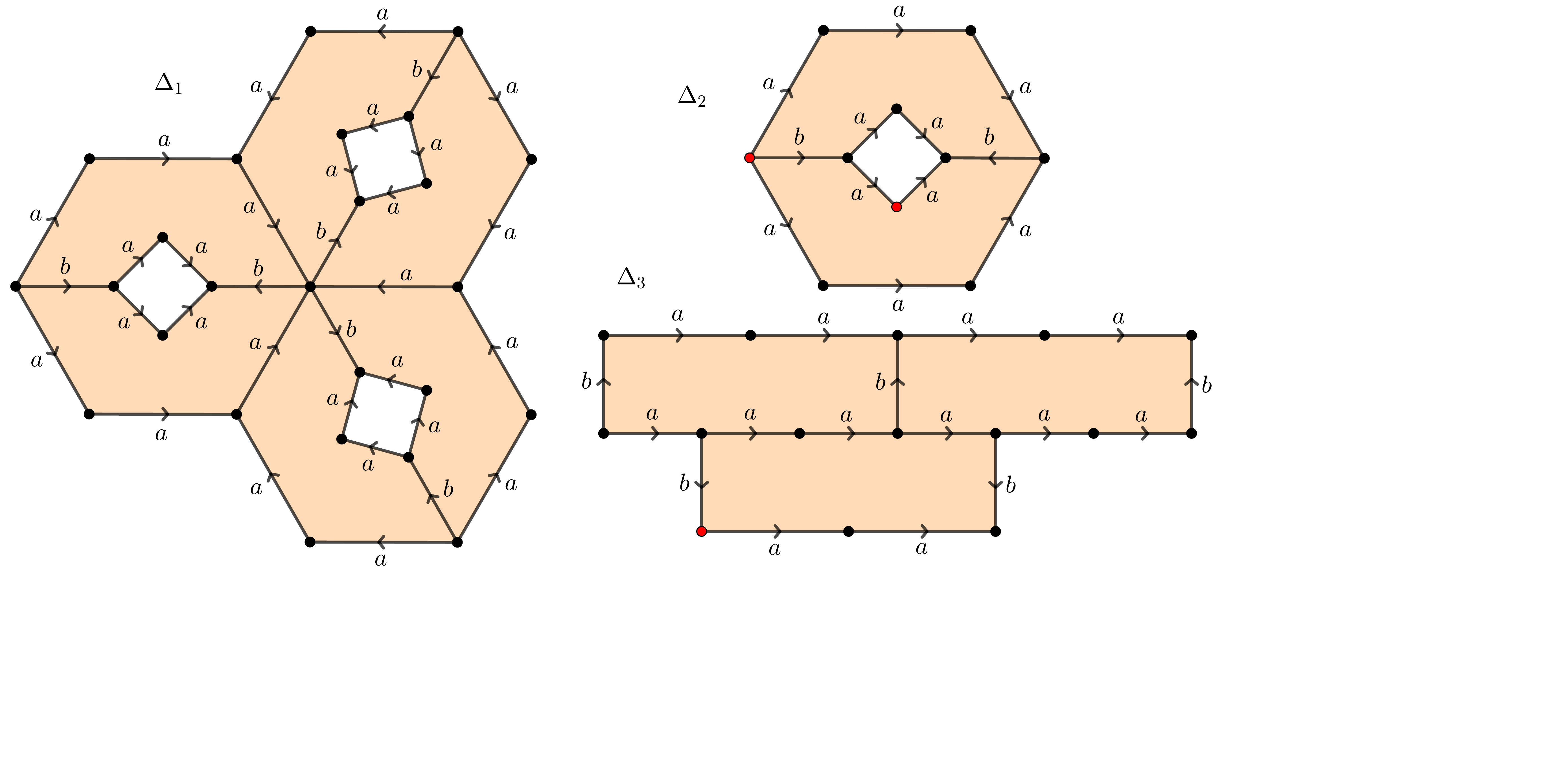}
\caption{A general diagram, a van Kampen diagram, and an annular diagram over the presentation $\langle a,b \mid ba^2b^{-1}=a^3 \rangle$.}
\label{diag}
\end{center}
\end{figure}

\noindent
We will be interested in two specific types of diagrams, namely van Kampen diagrams and annular diagrams, which are respectively related to the word and the conjugacy problems.

\begin{definition}
Let $\mathcal{P}= \langle X \mid R \rangle$ be a group presentation. A \emph{van Kampen diagram over $\mathcal{P}$} is a simply connected diagram over $\mathcal{P}$ with a fixed vertex in its boundary (i.e.\ the intersection between $D$ and the closure of $\mathbb{R}^2 \backslash D$). A \emph{boundary cycle} of $D$ is a cycle $\alpha$ of minimal length which contains all the edges in the boundary of $D$ which does not cross itself, in the sense that, if $e$ and $e'$ are consecutive edges of $\alpha$ with $e$ ending at a vertex $v$, then $e^{-1}$ and $e'$ are adjacent in the cyclically ordered set of all edges of $D$ beginning at $v$. The \emph{label} of the boundary of $D$ is the word labelling the boundary cycle of $D$ which begins at the basepoint of $D$ and which turns around $D$ clockwise. 
\end{definition}

\noindent
The connection between van Kampen diagrams and the word problem is made explicit by the following statement. We refer to \cite[Theorem V.1.1 and Lemma V.1.2]{LS} for a proof.

\begin{prop}
Let $\mathcal{P}= \langle X \mid R \rangle$ be a presentation of a group $G$ and $w \in X^{\pm}$ a non-empty word. There exists a van Kampen diagram over $\mathcal{P}$ whose boundary is labelled by $w$ if and only if $w=1$ in $G$.
\end{prop}

\noindent
Next, let us consider annular diagrams.

\begin{definition}\label{def:annulardiag}
Let $\mathcal{P}= \langle X \mid R \rangle$ be a group presentation. An \emph{annular diagrams over $\mathcal{P}$} is a diagram $D$ over $\mathcal{P}$ such that $\mathbb{R}^2 \backslash D$ has exactly two connected components, endowed with a fixed vertex in each connected component of its boundary (i.e.\ the  intersection between $D$ and the closure of $\mathbb{R}^2 \backslash D$). The \emph{inner boundary} (resp.\ \emph{outer boundary}) of $D$, denoted by $\partial_\text{inn}D$ (resp.\ $\partial_\text{out}D$), is the intersection of $D$ with the bounded (resp.\ unbounded) component of $\mathbb{R}^2 \backslash D$. A cycle of minimal length (that does not cross itself) which contains all the edges in the outer (resp.\ inner) boundary of $D$ is an \emph{outer} (resp.\ \emph{inner}) \emph{boundary cycle} of $D$. The \emph{label} of the boundary of $D$ is the couple $(w_1,w_2)$ where $w_1$ (resp.\ $w_2$) is the word labelling the inner (resp.\ outer) boundary cycle of $D$ which begins at the basepoint of $D$ and which turns clockwise.
\end{definition}

\noindent
The connection between annular diagrams and the conjugacy problem is made explicit by the following statement. We refer to \cite[Lemmas V.5.1 and V.5.2]{LS} for a proof.

\begin{prop}
Let $\mathcal{P}= \langle X \mid R \rangle$ be a presentation of a group $G$ and $w_1,w_2 \in X^{\pm}$ two non-empty words. There exists an annular diagrams over $\mathcal{P}$ whose boundary is labelled by $(w_1,w_2)$ if and only if $w_1$ and $w_2$ are conjugate in $G$.
\end{prop}

\noindent
When our group presentation contains only relations of length four, diagrams are made of squares. This allows us to introduce \emph{dual curves}, i.e.\ non-empty minimal subsets $\alpha$ such that the intersection between $\alpha$ and an edge is always either empty or the midpoint and such that if $\alpha$ contains the midpoint of one side of a square then it contains the segment connecting this point to the midpoint of the opposite side. Dual curves in diagrams play the same role as hyperplanes in median graphs. The following result, which is a direct consequence of \cite[Corollary~2.4]{MR4298722}, will play a fundamental role in our solution to the conjugacy problem:

\begin{prop}\label{prop:AnnularDiag}
Let $\mathcal{P}= \langle \Sigma \mid \mathcal{R} \rangle$ be a group presentation whose Cayley complex satisfies the assumption of Theorem~\ref{thm:CriterionCC}. Let $w \in \Sigma^\pm$ be a word and $\Delta$ a van Kampen diagram whose boundary is labelled by $w$. If $\Delta$ contains a dual curve that is not an embedded segment or two dual curves intersecting more than once, then there exists another van Kampen diagram whose boundary is also labelled by $w$ but whose area is $\leq \mathrm{area}(\Delta)-2$. 
\end{prop}

\noindent
We are finally ready to describe how to solve the conjugacy problem in cactus groups. In our next statement, a \emph{simple-conjugation} applied to a word $w=\ell_1 \cdots \ell_r$ refers to the following operation: given $1 \leq a < b \leq n$ such that the permutation of $\{1, \ldots, n\}$ associated to $w$ stabilises $[a,b]$, replace $w$ with the word obtained from $s_{a,b}w$ by shifting $s_{a,b}$ all the way to the end of $w$ and by removing the final letter $s_{a,b}$. Observe that the word thus obtained equals $s_{a,b}ws_{a,b}^{-1}$ in the cactus group and has the same length as $w$. Figure~\ref{SimpleConj} gives an example of two diagrams related by a simple-conjugation.
\begin{figure}[h!]
\begin{center}
\includegraphics[width=0.6\linewidth]{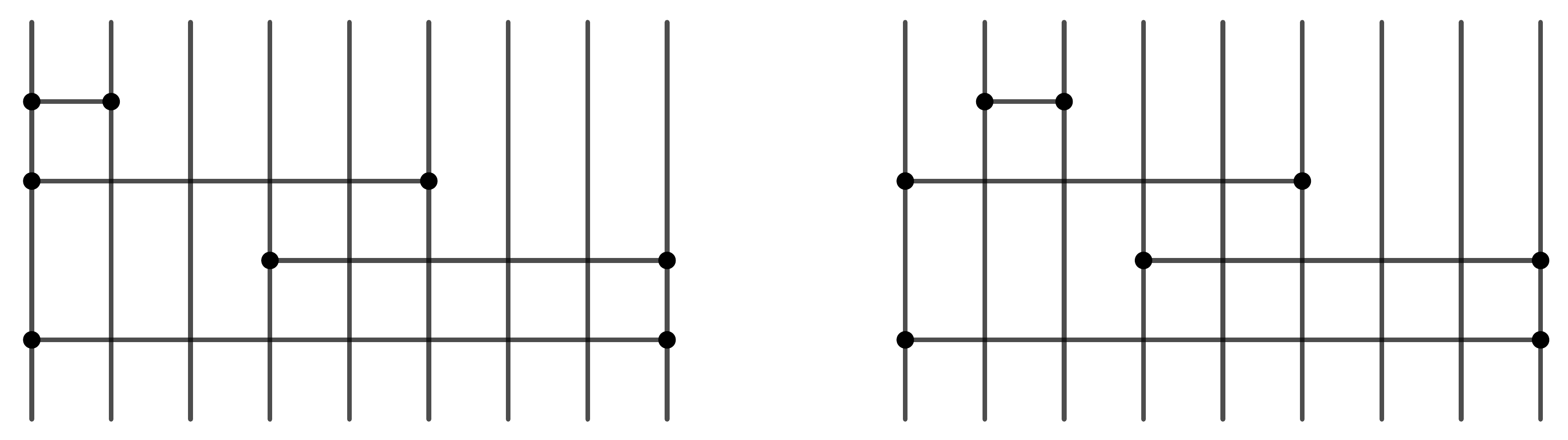}
\caption{Two diagrams conjugate by $s_{1,3}$. }
\label{SimpleConj}
\end{center}
\end{figure}

\noindent
We also refer to a \emph{cyclically irreducible} word $w$ as an irreducible word that does not contain two letters $\ell,\ell'$ such that $\ell$ can be shifted all the way to the beginning of the word, becoming $\ell''$ in the process, and such that $\ell'$ can be shifted all the way to the end of the word, becoming $\ell''$ in the process. Notice that, if there exist such letters $\ell$ and $\ell'$, then we can obtain from $w$ by mock commutations a word of the form $\ell' w' \ell'$. Then $w'$ is shorter that $w$ but remains in the same conjugacy class in the cactus group. Thus, in order to solve the conjugacy problem in the whole group, it suffices to determine when two given cyclically irreducible words represent the same element up to conjugacy.

\begin{thm}
Let $n \geq 2$ be an integer. Two cyclically irreducible words of generators $w_1$ and $w_2$ represent conjugate elements in the cactus group $J_n$ if and only if there exist $w_1'$ and $w_2'$ obtained from $w_1$ and $w_2$ by mock commutations such that a cyclic shift of $w_2'$ can be obtained from $w_1'$ by simple-conjugations.
\end{thm}

\begin{proof}
Let $w_1$ and $w_2$ be two irreducible words. Assume that they represent conjugate elements in $J_n$. According to Proposition~\ref{prop:AnnularDiag}, for all words $w_1',w_2'$ obtained from $w_1,w_2$ by mock commutation, there exists an annular diagram $\Delta$ whose boundary is labelled by $(w_1',w_2')$. We choose $w_1'$, $w_2'$, and $\Delta$ in order to minimise the area of $\Delta$. 

\begin{claim}
Dual curves in $\Delta$ do not self-intersect.
\end{claim}

\noindent
Each edge of $\Delta$ is labelled by a generator. Even though two opposite edges in a square may be labelled by distinct generators, the two generators inverse the same set of strands. This allows us to label the dual curves of $\Delta$ by sets of strands. Notice that, in a square, the two pairs of opposite edges are labelled by properly nested or disjoint sets of strands. This implies that:

\begin{fact}\label{fact:TransverseCurves}
In $\Delta$, transverse dual curves are labelled by properly nested or disjoint sets of strands. 
\end{fact}

\noindent
As a consequence, a dual curve cannot self-intersect. 

\begin{claim}
There exist only two types of dual curves in $\Delta$: embedded segments connecting the inner and outer boundaries, and embedded circles separating the inner and outer boundaries. 
\end{claim}

\noindent
Because dual curves do not self-intersect, they must be embedded segments or embedded circles. It remains to show that a dual curve cannot cross twice the inner or outer boundary and that a circular dual curve must separate the inner and outer boundaries.

\medskip \noindent
The latter assertion is rather straightforward. A circular dual curve that does not separate the inner and outer boundaries delimits a subdisc, which we can think about as a van Kampen diagram $\Omega$. According to Proposition~\ref{prop:AnnularDiag}, we can find another van Kampen diagram $\Omega'$ with the same boundary but with smaller area. Replacing $\Omega$ in $\Delta$ with $\Omega'$ shortens the area of $\Delta$, which is impossible.

\medskip \noindent
Next, assume for contradiction that there exists dual curve $\alpha$ crossing the inner or outer boundary twice, say along the edges $e_1,e_2$. We choose $e_1$ and $e_2$ as close as possible, which implies that no dual curve crosses the boundary between $e_1$ and $e_2$. As a consequence, the dual curves crossing the boundary between $e_1$ and $e_2$ intersect $\alpha$. It follows from Fact~\ref{fact:TransverseCurves} that the set of strands labelling $e_1,e_2$ is nested or disjoint with all the sets of strands labelling the edges of the boundary lying between $e_1$ and $e_2$. This implies that, in the (cyclic) word labelling the boundary, the letter given by $e_1$ can be mock commuted all the way to the letter given by $e_2$ and in the process becomes identical to the latter. This contradicts the fact that $w_1$ and $w_2$ are cyclically irreducible. 

\begin{claim}\label{claim:SameType}
Two dual curves of the same type cannot intersect. 
\end{claim}

\noindent
Assume for contradiction that there exist two transverse dual curves $\alpha,\beta$ that connect the inner and outer boundaries. Then there a subdisc delimited by $\alpha$, $\beta$, and a subsegment $\gamma$ of the outer boundary. Because every dual curve intersecting $\gamma$ has to cross $\alpha$ or $\beta$, we can choose $\alpha$ and $\beta$ so that $\gamma$ is reduced to a single vertex. Let $e_1,e_2$ denote the adjacent edges crossed by $\alpha,\beta$. Because $\alpha$ and $\beta$ are transverse, the generators labelling $e_1$ and $e_2$ conjugacy-commute. Therefore, we can glue a new square to $\Delta$ along $e_1 \cup e_2$. The new annular diagram $\Delta'$ has its outer boundary labelled by a word obtained from $w_1'$ by applying a mock commutation; its area is $\mathrm{area}(\Delta)+1$, but its contains a subdisc with two dual curves intersecting twice, so its area can be shortened by two according to Proposition~\ref{prop:AnnularDiag}. This contradicts the minimality of the area of $\Delta$.
\begin{center}
\includegraphics[width=0.7\linewidth]{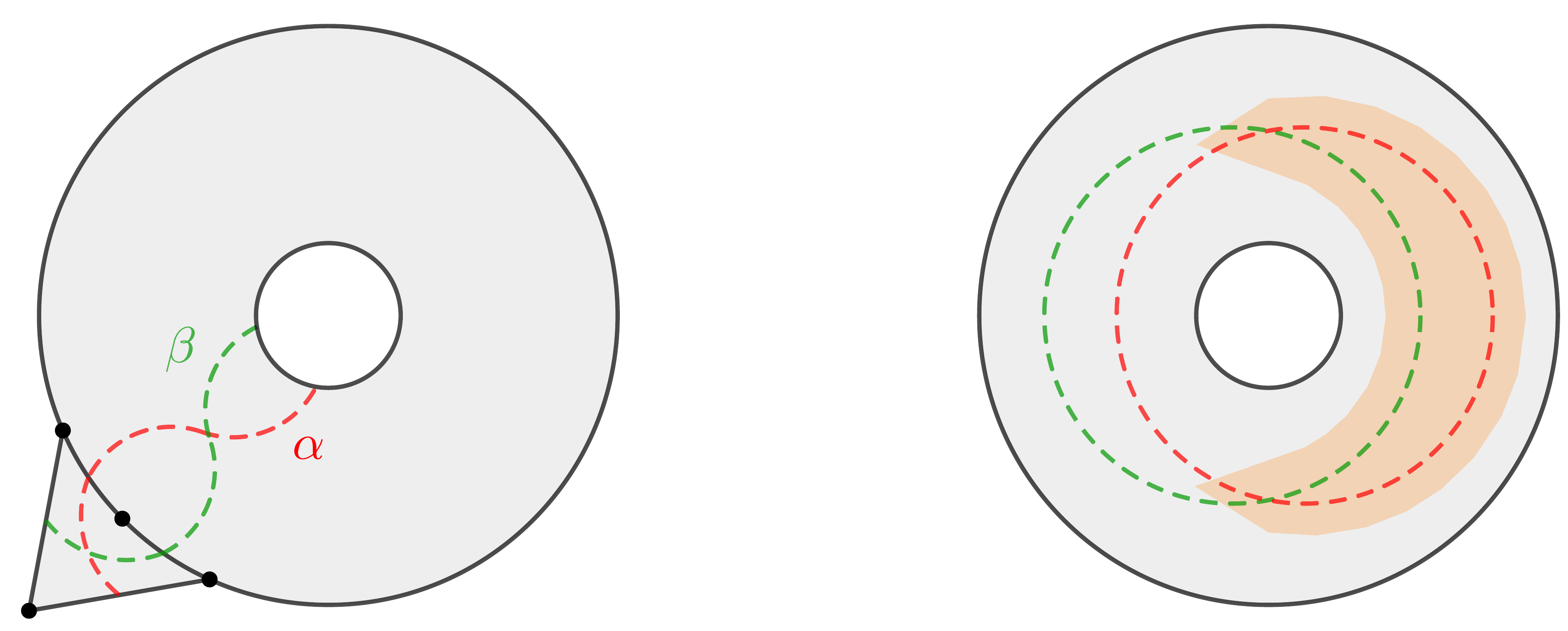}
\end{center}

\noindent
Next, assume that two circular dual curves intersect. If so, they must intersect at least twice. But then we can find a subdisc with two intersecting dual curves, so it follows from Proposition~\ref{prop:AnnularDiag} that the area of this subdisc (and a fortiori of $\Delta$) can be shortened. Again, this contradicts the minimality of the area of $\Delta$. This concludes the proof of Claim~\ref{claim:SameType}. 

\medskip \noindent
Because $\Delta$ coincides with the complex dual to the net of its dual curves, it follows from our previous observations that $\Delta$ is combinatorially a product of a (subdivided) circle with a (subdivided) interval. Let $C_1, \ldots, C_r$ denote the concentric circles contained in $\Delta$, starting from the outer boundary $C_1$ and ending with the inner boundary $C_r$. Fixing on each such circle the projection of the basepoint of $C_1$, we obtain a sequence of words $m_1, \ldots, m_r$ labelling $C_1, \ldots, C_r$. Notice that $m_1=w_1'$ and that $m_r$ is a cyclic shift of $w_2'$. In order to conclude, it suffices to show that $m_i$ is obtained from $m_{i-1}$ by a simple-conjugation for every $2 \leq i \leq r$. 

\medskip \noindent
Fix an index $2 \leq i \leq r$ and let $s$ be the generator labelling the edge $e$ which connects the two basepoints of $C_i$ and $C_{i-1}$. Let $\kappa_j$ denote the path that starts at the basepoint of $C_i$, follows $C_i$ along $j$ edges, crosses an edge connecting $C_i$ and $C_{i-1}$, and follows $C_{i-1}$ until the basepoint of $C_{i-1}$. Notice that $\kappa_0$ is labelled by $sm_{i-1}$, that $\kappa_{|m_i|}$ is labelled by $m_is$, and that the sequence of words labelling the $\kappa_j$ amounts to shifting $s$ all the way to the end of $m_{i-1}$ in $sm_{i-1}$. During the process, the letter $s$ can vary but it ends as $s$ again, which amounts to saying that the interval associated to the generator $s$ is fixed by the permutation associated to $m_{i-1}$. Thus, we have proved that $m_i$ is obtained from $m_{i-1}$ by a simple-conjugation, as desired. 
\end{proof}

\section{Actions of the pure cactus groups}\label{section:Pure}

\noindent
Admitting a proper and cocompact action on a median graph already provides a lot of valuable information. As shown in \cite{MR2377497} (see also \cite{QMspecial}), when the action avoids some pathological configurations regarding (orbits of) hyperplanes, it is possible to restrict even further the structure of the group under consideration. 

\begin{definition}
Let $G$ be a group acting on a median graph $X$.
\begin{itemize}
	\item A \emph{self-intersection} is the data of a hyperplane $H$ of $X$ and of an element $g \in G$ such that $H$ and $gH$ are transverse.
	\item A \emph{self-osculation} is the data of a hyperplane $H$ of $X$ and of an element $g \in G$ such that $H$ and $gH$ are tangent.
	\item An \emph{inter-osculation} is the data of two hyperplanes $H_1,H_2$ of $X$ and of an element $g \in G$ such that $H_1$ is transverse to $H_2$ and $gH_1$ tangent to $H_2$.
\end{itemize}
\end{definition}

\noindent
The motivation for this definition is the following. Let $G$ be a group acting on a median graph $X$. We fix a vertex $o \in X$ and we denote by $\Gamma$ the graph whose vertices are the $G$-orbits of hyperplanes in $X$ and whose edges connect two orbits whenever they contain two transverse hyperplanes. Consider the map from $G$ to the right-angled Coxeter group $C(\Gamma)$ defined by
$$\Omega : g \mapsto \text{word of the (orbits of) hyperplanes crossing a geodesic } [o,go].$$
It can be shown that $\Omega$ is well-defined, i.e.\ the element of $C(\Gamma)$ we get does not depend on the geodesic $[o,go]$ we choose. Moreover, $\Omega$ is a morphism. However, this morphism may be trivial, some restriction are needed in order to make the morphism interesting. As shown in \cite{MR2377497} (see also \cite{QMspecial}), if the action of $G$ on $X$ is faithfull, has no self-intersection, no self-osculation, and no inter-osculation, then the morphism $\Omega$ is injective. 

\medskip \noindent
The action of a cactus group $J_n$ on its median graph $\mathscr{C}_n$ does not satisfy these conditions. In fact, $J_n$ does not admit such an action on any median graph. Indeed, $J_n$ (for $n \geq 4$) cannot be embedded into a right-angled Coxeter group since it contains a non-abelian finite subgroup (namely $(\mathbb{Z}/2\mathbb{Z} \oplus \mathbb{Z}/2 \mathbb{Z} ) \rtimes \mathbb{Z} / 2 \mathbb{Z}$ where the right factor $\mathbb{Z}/2 \mathbb{Z}$ acts on the direct sum by permuting the two factors). Nevertheless:

\begin{thm}\label{thm:conspicial}
Let $n \geq 2$. The action of the pure cactus group $PJ_n$ on $\mathscr{C}_n$ has no self-intersection, no self-osculation, no inter-osculation, and has trivial cube-stabilisers.
\end{thm}

\noindent
In order to prove the theorem, we need to labels hyperplanes in our median graphs. Given an $n \geq 2$, an oriented edge $e$ in $\mathscr{C}_n$ can be uniquely written as $(g,gs_{i,j})$ for some $g \in J_n$ and $1 \leq i < j \leq n$. Roughly speaking, we start with a braided-like picture $g$ and we braid some of its strands by applying $s_{i,j}$ at the bottom of $g$. We would like to label $e$ with the set of the strands of $g$ which are braided by $s_{i,j}$. Formally, this is the subset $\Sigma(g)([i,j])$ of $\{1, \ldots, n\}$. Observe that the inverse $-e$ of $e$ has the same label. Indeed, $-e$ can be written as $(gs_{i,j},gs_{i,j} \cdot s_{i,j})$ so its label is 
$$\Sigma(gs_{i,j})([i,j]) = \Sigma(g)\left( \Sigma(s_{i,j}) ([i,j]) \right) = \Sigma(g)([i,j]),$$
since $\Sigma(s_{i,j})$ stabilises $[i,j]$. Thus, our labelling actually labels the (unoriented) edges of $\mathscr{C}_n$. The key observation is that this labelling extends to a labelling of the hyperplanes.

\begin{lemma}\label{lem:LabelsHyp}
Let $n \geq 2$. The edges of a given hyperplane are all labelled by the same set of strands. Moreover, if $H$ is a hyperplane labelled by a set of strands $S$ and if $g \in J_n$ is an element, then $gH$ is labelled by $\Sigma(g)(S)$. 
\end{lemma}

\begin{proof}
Let $H$ be a hyperplane and $g \in J_n$ an element. Fix an edge $e$ in $H$. It can be written as $(h,hs_{i,j})$ for some $h \in J_n$ and $1 \leq i < j \leq n$. So $H$ is labelled by the set of strands $S:=\Sigma(h)([i,j])$. Because the edge $(gh,ghs_{i,j})$ belongs to $gH$, it follows that $gH$ is labelled by the set of strands
$$\Sigma(gh)([i,j]) = \Sigma(g) \left( \Sigma(h)([i,j]) \right) = \Sigma(g)(S),$$
as desired. This proves the second assertion of our lemma. 

\medskip \noindent
In order to justify the first assertion of our lemma, it suffices to notice that two opposite edges in a $4$-cycle are labelled by the same set of strands. Corollary~\ref{cor:Cycles} describes the $4$-cycles in $\mathscr{C}_n$ and it follows from the first paragraph of the proof that we can choose representatives of our $4$-cycles modulo $J_n$-translations. Consequently, it suffices to verify our assertion on the $4$-cycles of the following forms:
\begin{itemize}
	\item $(1, s_{p,q}, s_{p,q} s_{m,r}, s_{p,q} s_{m,r} s_{p,q}, s_{p,q}s_{m,r}s_{p,q} s_{m,r}=1)$ for some $1 \leq p < q \leq n$ and $1 \leq m < r \leq n$ satisfying $[p,q] \cap [m,r]= \emptyset$;
	\item $(1, s_{p,q}, s_{p,q}s_{m,r}, s_{p,q}s_{m,r}s_{p,q}, s_{p,q}s_{m,r} s_{p,q} s_{p+q-r,p+q-m}=1)$ for some $1 \leq p < q \leq n$ and $1 \leq m < r \leq n$ satisfying $[m,r] \subset [p,q]$. 
\end{itemize}
In the first case, the edges $(1,s_{p,q})$ and $(s_{m,r},s_{m,r}s_{p,q})$ are both labelled by $[p,q]$, and the edges $(1,s_{m,r})$ and $(s_{p,q},s_{m,r})$ are both labelled by $[m,r]$. In the second case, the edges $(1,s_{p,q})$ and $(s_{m,r}, s_{m,r}s_{p,q})$ are both labelled by $[p,q]$, and the edges $(1,s_{m,r})$ and $(s_{p,q},s_{p,q}s_{p+q-m,p+q-r})$ are both labelled by $[m,r]$. 
\end{proof}

\noindent
The following observation will be also needed in order to prove Theorem~\ref{thm:conspicial}. 

\begin{lemma}\label{lem:HypTransverseLabel}
Let $o \in \mathscr{C}_n$ be a vertex and $e_1,e_2$ two edges having $o$ as an endpoint. The hyperplanes containing $e_1$ and $e_2$ are transverse if and only if they are labelled by disjoint or properly nested sets of strands.
\end{lemma}

\begin{proof}
Let $H_1$ (resp.\ $H_2$) denote the hyperplane containing $e_1$ (resp.\ $e_2$). Then $H_1$ and $H_2$ are transverse if and only if $e_1$ and $e_2$ span a $4$-cycle. The desired conclusion follows from the description of $4$-cycles in $\mathscr{C}_n$ given by Corollary~\ref{cor:Cycles}.
\end{proof}

\begin{proof}[Proof of Theorem~\ref{thm:conspicial}.]
It follows from Lemma~\ref{lem:LabelsHyp} that $PJ_n$ acts on $\mathscr{C}_n$ by preserving the labels of the hyperplanes. Thus, we deduce from Lemma~\ref{lem:HypTransverseLabel} that the action of $PJ_n$ on $\mathscr{C}_n$ has no self-intersection, no self-osculation, and no inter-osculation. 

\medskip \noindent
Let $g \in J_n$ be a non-trivial element stabilising some cube $C$ in $\mathscr{C}_n$. Up to conjugating $g$ with an element of $J_n$, we can assume without loss of generality that $1$ belongs to $C$, which implies that $g$ (as a vertex) belongs to $C$ as well. Let $S_1, \ldots, S_r$ denote the labels of the hyperplanes separating $1$ and $g$. Because the hyperplanes crossing $C$ are pairwise transverse, we know from Lemma~\ref{lem:HypTransverseLabel} that the $S_i$ are pairwise disjoint or properly nested. Up to re-indexing our collection, we assume that $S_i \subsetneq S_j$ implies $i>j$. Because crossing the hyperplane labelled by $S_i$ amounts to inverting the strands in $S_i$, we can describe $\Sigma(g)$ as the permutation obtained by inverting the strands in $S_1$, next the strands in $S_2$, and so on until $S_r$. We write this decomposition as $\Sigma(g)= \sigma_r \cdots \sigma_1$. 

\medskip \noindent
Let $a \in \{1, \ldots, n\}$ denote the leftmost element in $S_1$. Then $\sigma_1(a)$ is the rightmost element in $S_1$. If $S_2$ does not contain $\sigma_1(a)$, then $\sigma_2\sigma_1(a)= \sigma_1(a) \neq a$. Otherwise, $\sigma_2\sigma_1(a)$ belongs to $S_2$, which must be properly contained in $S_1$. A fortiori, $\sigma_2 \sigma_1(a) \neq a$. We can keep going. If $\sigma_2\sigma_1(a)$ does not belong to $S_2$, then $\sigma_3\sigma_2\sigma_1(a)= \sigma_2\sigma_1(a) \neq a$. Otherwise, $\sigma_3\sigma_2\sigma_1(a)$ belongs to $S_3$, which is properly contained in $S_2$, which implies that $\sigma_3\sigma_2\sigma_1(a) \neq a$. Eventually, we conclude that $\Sigma(g)(a) \neq a$, hence $\Sigma(g) \neq 1$ or equivalently $g \notin PJ_n$. 
\end{proof}

\noindent
As a first application of Theorem~\ref{thm:conspicial}, we recover \cite[Theorem~C]{Paolo}:

\begin{cor}\label{cor:PureTorsionFree}
Pure cactus groups are torsion-free.
\end{cor}

\begin{proof}
Because an isometry of finite order in a median graph has to stabilise a cube, the desired conclusion follows from Theorem~\ref{thm:conspicial}. 
\end{proof}

\noindent
As a second application, it follows that pure cactus groups embeds into right-angled Coxeter groups. However, the embedding described above is often not optimal: the right-angled Coxeter group is usually ``too big''. A solution to this problem is to axiomatise the relevant properties of the colouring of the hyperplanes of our median graphs as orbits under the group action.

\begin{thm}\label{thm:Special}
Let $G$ be a group acting on a median graph. Fix colouring $\mathfrak{c}$ of the hyperplanes of $X$ and $\Gamma$ a simple graph whose vertices are the possible colours of a hyperplane. Assume that:
\begin{itemize}
	\item $\mathfrak{c}$ is $G$-equivariant, i.e.\ $\mathfrak{c}(gA)=\mathfrak{c}(A)$ for every $g \in G$ and every hyperplane~$A$.
	\item any two hyperplanes whose carriers intersect are equal if and only if they have the same colours;
	\item any two hyperplanes whose carriers intersect are transverse if and only if their colours are adjacent in $\Gamma$.
\end{itemize}
Given a basepoint $o \in X$, the map  
$$\Omega : g \mapsto \begin{array}{c} \text{word of colours given by the hyperplanes} \\ \text{successively crossed by a path from $o$ to $go$} \end{array}.$$
induces an injective morphism $G \hookrightarrow C(\Gamma)$. 
\end{thm}

\noindent
We refer to the appendix for a proof of this statement. Now, let us apply this construction to pure cactus groups. Given an $n \geq 2$, we colour each hyperplane of $\mathscr{C}_n$ by its set of strands, and we let $\Gamma$ be the graph whose vertices are the subsets of $\{1, \ldots, n\}$ and whose edges connect any two subsets whenever they are disjoint or nested. The assumptions of Theorem~\ref{thm:Special} are satisfied for the action of $PJ_n$ as a consequence of Lemmas~\ref{lem:LabelsHyp} and~\ref{lem:HypTransverseLabel}. Thus, the map
$$\Omega : g \mapsto \begin{array}{c} \text{word of labels given by the hyperplanes} \\ \text{successively crossed by a path from $1$ to $g$} \end{array}.$$
induces an injective morphism $PJ_n \hookrightarrow C(\Gamma)$. The image of an element can be described more explicitly. Let $g \in PJ_n$ be an arbitrary element and write it as a word of generators $s_{p_1,q_1} \cdots s_{p_r,q_r}$. Then
$$1, s_{p_1,q_1}, s_{p_1,q_1}s_{p_2,q_2}, \ldots, s_{p_1,q_1} \cdots s_{p_r,q_r}$$
defines a path in $\mathscr{C}_n$ from $1$ to $g$. For every $1 \leq i \leq r$, the $i$th edge of this path is labelled by the set of strands
$$\Sigma(s_{p_1,q_1} \cdots s_{p_{i-1},q_{i-1}}) ([p_i,q_i]).$$

\noindent
\begin{minipage}{0.28\linewidth}
\begin{center}
\includegraphics[width=0.65\linewidth]{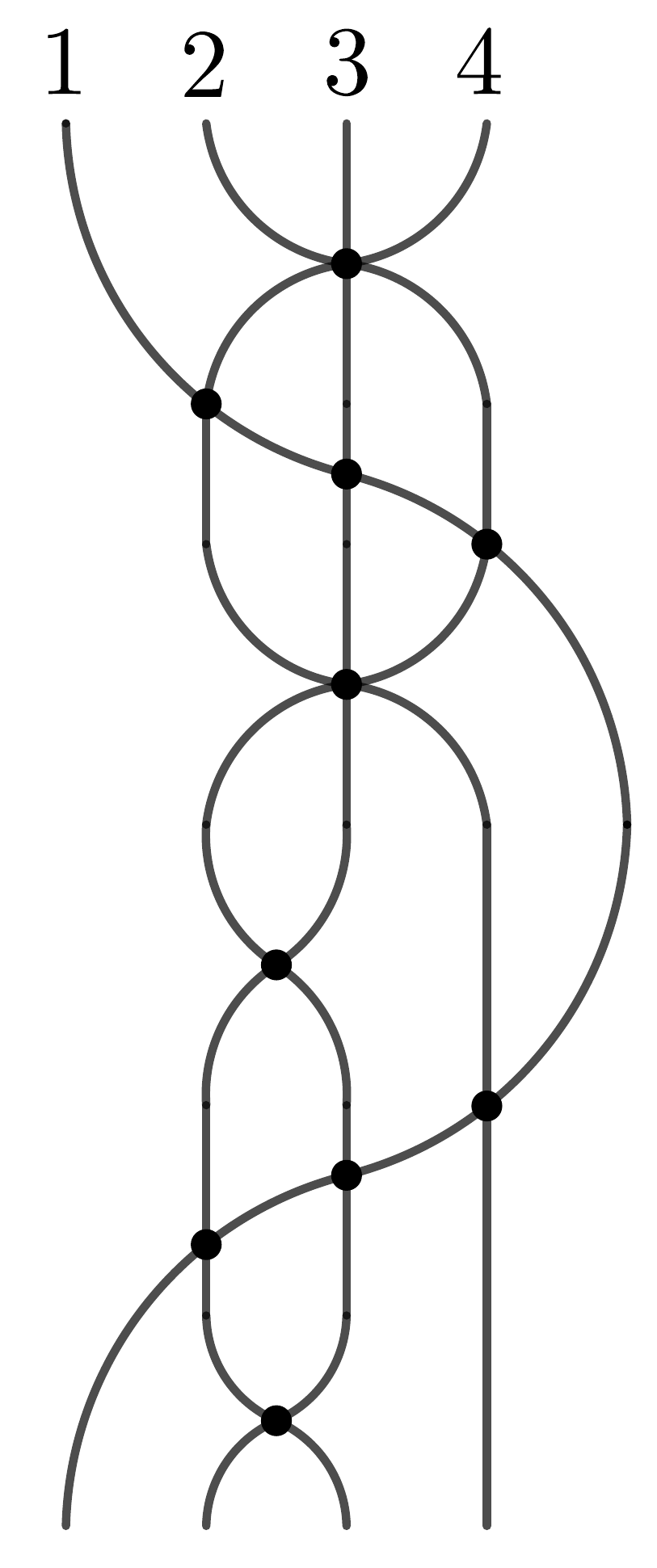}
\end{center}
\end{minipage}
\begin{minipage}{0.7\linewidth}
In other words, each generator $s_{p_i,q_i}$ in the word is sent to the set of the strands which cross the nodes corresponding to $s_{p_i,q_i}$ in the braided-like picture. For instance, the element fo $PJ_4$ illustrated by the braided-like picture on the left is sent to the word of labels $$\{2,3,4\} \{1,4\} \{1,3\} \{1,2\} \{2,3,4\} \{2,3\} \{1,4\} \{ 1,2\} \{ 1,3\} \{2,3\}$$
Interestingly, it seems that our embeddings of pure cactus groups in specific right-angled Coxeter groups coincide with the main construction of \cite{MR3988817}. 
\end{minipage}

\begin{remark}
It is worth noticing that our map $\Omega : PJ_n \to C(\Gamma)$ makes sense on the whole cactus group $J_n$. Even though this extension is no longer a morphism, it remains injective, which can be useful. The injectivity can be obtained by reproducing the proof of Theorem~\ref{thm:Special} given in the appendix word for word. 
\end{remark}

\section{Semidirect product decompositions}

\subsection{From top to bottom}

\noindent
For every $n \geq 2$ and every $S \subset \{2, \ldots, n\}$, let $J_n^S$ denote the group defined by the generators $s_{p,q}$, where $1 \leq p < q \leq n$ satisfy $q-p+1 \in S$, and by the relations:
\begin{itemize}
	\item $s_{p,q}^2=1$ for all $1 \leq p<q \leq n$ satisfying $q-p+1 \in S$;
	\item $s_{p,q}s_{m,r} = s_{m,r}s_{p,q}$  for all $1 \leq p<q \leq n$ and $1 \leq m < r \leq n$ satisfying $q-p+1,r-m+1 \in S$ and $[p,q] \cap [m,r] = \emptyset$;
	\item $s_{p,q}s_{m,r} = s_{p+q-r,p+q-m}s_{p,q}$ for  all $1 \leq p<q \leq n$ and $1 \leq m < r \leq n$ satisfying $q-p+1,r-m+1 \in S$ and $[m,r] \subset [p,q]$.
\end{itemize}
In other words, starting from $J_n$, we restrict the number of strands that can meet at a node in a braided-like picture. Of course, $J_n^{[2,n]}$ coincides with $J_n$. The group $J_n^{\{2\}}$ coincides with the so-called \emph{twin group} or \emph{graph planar group} $T_n$. %In order to shorten the notation, we write $J_n^a$ instead of $J_n^{\{a\}}$. 

\medskip \noindent
It follows from Proposition~\ref{prop:Normal} that the map sending each generator $s_{p,q}$ of $J_n^S$ to the generator $s_{p,q}$ of $J_n$ induces an injective morphism. In the sequel, we will identify $J_n^S$ with its image in $J_n$. 

\medskip \noindent
These subgroups of cactus groups turn out to be nicely embedded, as motivated by our next statement. Notice, however, that it already follows from Section~\ref{section:Word} that $J_n^S$ isometrically embeds into $J_n$ when both groups are endowed with the word lengths given by their canonical generating sets. 

\begin{prop}\label{prop:ConvexCocompact}
Fix an $n \geq 2$ and a subset $S \subset \{2, \ldots, n\}$. With respect to the action of $J_n$ on $\mathscr{C}_n$, $J_n^S$ is convex-cocompact. In particular, $J_n^S$ is quasi-isometrically embedded in $J_n$. 
\end{prop}

\begin{proof}
Let $\mathscr{C}_n^S$ denote the subgraph of $\mathscr{C}_n$ given by the vertices representing elements in $J_n^S$. This is a natural copy of the Cayley graph of $J_n^S$ in $\mathscr{C}_n$. Clearly, $J_n^S$ acts on $\mathscr{C}_n^S$ cocompactly. Moreover, because geodesics in $\mathscr{C}_n$ are labelled by words of minimal lengths, or equivalently by irreducible words, we know that every geodesic between two vertices $a,b \in \mathscr{C}_n^S$ corresponds to an irreducible word representing $a^{-1}b$. But we know from Section~\ref{section:Word} that every irreducible representing $a^{-1}b$ is written only using the generators of $J_n^S$, so all the geodesics connecting $a$ and $b$ must stay in $\mathscr{C}_n^S$. Thus, $\mathscr{C}_n^S$ is a convex subgraph of $\mathscr{C}_n$. 
\end{proof}

\begin{cor}
Fix an $n \geq 2$ and a subset $S \subset \{2, \ldots, n\}$. The subgroup $PJ_n^S:= J_n^S \cap PJ_n$ is a virtually retract, and in particular separable, in $PJ_n$.
\end{cor}

\begin{proof}
By combining Theorem~\ref{thm:conspicial} with \cite{MR2377497}, we know that convex-cocompact subgroups of $PJ_n$ are virtual retracts. Thus, the desired conclusion follows from Proposition~\ref{prop:ConvexCocompact}. 
\end{proof}

\noindent
It follows directly from the presentation of our cactus group $J_n$ that
$$\begin{array}{lcl} J_n & = & J_n^{[2,n-1]} \rtimes J_n^{\{n\}}  \\ \\ & = & \left( J_n^{[2,n-2]} \rtimes J_n^{\{n-1\}} \right) \rtimes J_n^{\{n\}} \\ & \vdots & \\ & =& \left( \left( \cdots \left( J_n^{\{2\}} \rtimes J_n^{\{3\}} \right) \rtimes \cdots \right) \rtimes J_n^{\{n-1\}} \right) \rtimes J_n^{\{n\}} \end{array}$$
This is the \emph{top-to-bottom decomposition} of the cactus group $J_n$. Notice that each $J_n^{\{i\}}$ has only commutations in its presentation, so all our factors are right-angled Coxeter groups. Because finite subgroups in right-angled Coxeter groups are powers of $\mathbb{Z}/2\mathbb{Z}$, it immediately follows that:

\begin{cor}[\cite{Paolo}]\label{cor:Torsion}
In cactus groups, finite subgroups are $2$-groups. 
\end{cor}

\noindent
The precise structure of subgroups in cactus groups remains unclear. See Section~\ref{section:Open}.

\subsection{From bottom to top}

\noindent
In a median graph, a \emph{reflection} is an isometry that stabilises all the edges of a hyperplane and switches the two halfspaces it delimits. It turns out that, if a group acts on a median graph and contains at least one reflection, then it naturally decomposes as a semi-direct product. In particular, a subgroup generated by reflections is always isomorphic to a right-angled Coxeter group.

\begin{thm}[\cite{MR4359526}]\label{thm:ReflectionGeneral}
Let $G$ be a group acting on a median graph $X$ with trivial vertex-stabilisers. Assume that $\mathcal{J}$ is a $G$-invariant collection of hyperplanes such that $G$ contains a reflection $r_J$ along each hyperplane $J$ of $\mathcal{J}$. If $Y \subset X$ denotes the intersection of all the halfspaces which are delimited by hyperplanes of $\mathcal{J}$ and which contain a fixed basepoint $x_0 \in X$, then  
$$G = R \rtimes \mathrm{stab}(Y), \ \text{where} \ R = \langle r_J, \ J \in \mathcal{J} \rangle.$$ 
Moreover, $Y$ is a fundamental domain of $R \curvearrowright X$ and, if $\mathcal{J}_0$ denotes the maximal $x_0$-peripheral subcollection of $\mathcal{J}$ and if $\Delta$ denotes its crossing graph, then the map sending a vertex $J$ of $\Delta$ to the reflection $r_J$ of $R$ induces an isomorphism $C(\Delta) \to R$. 
\end{thm}

\noindent
A collection of hyperplanes is \emph{$x_0$-peripheral} if it does not contain two hyperplanes such that one separates the other from $x_0$. The \emph{crossing graph} of a collection of hyperplanes is the graph whose vertices are the hyperplanes of the collection and whose edges connect two hyperplanes whenever they are transverse. 

\medskip \noindent
Given the action of our cactus group $J_n$ on its median graph $\mathscr{C}_n$, the canonical generators of $J_n$ all invert edges, but only some of them are reflections. Nevertheless, these reflections lead to interesting decompositions of cactus groups. 

\begin{lemma}\label{lem:Reflection}
Fix $n \geq 2$ and $S \subset \{2, \ldots, n\}$. Let $\mathscr{C}_n^S$ denote the subgraph of $\mathscr{C}_n$ whose vertices represent elements in $J_n^S$. With respect to the action of $J_n^S$ on $\mathscr{C}_n^S$, every generator $s_{p,q}$ with $q-p+1= \min(S)$ is a reflection.
\end{lemma}

\begin{proof}
Let $e$ be an oriented edge in the same oriented hyperplane $H$ as $(1,s_{p,q})$. Write $e=(g,gs_{a,b})$. Fixing a geodesic from $1$ to $g$, which necessarily lies in a fibre of $H$ according to Theorem~\ref{thm:MedianBig}, let $\ell_1 \cdots \ell_r$ denote the corresponding word of generators which represents $g$. Geometrically, the configuration is:
\begin{center}
\includegraphics[width=0.6\linewidth]{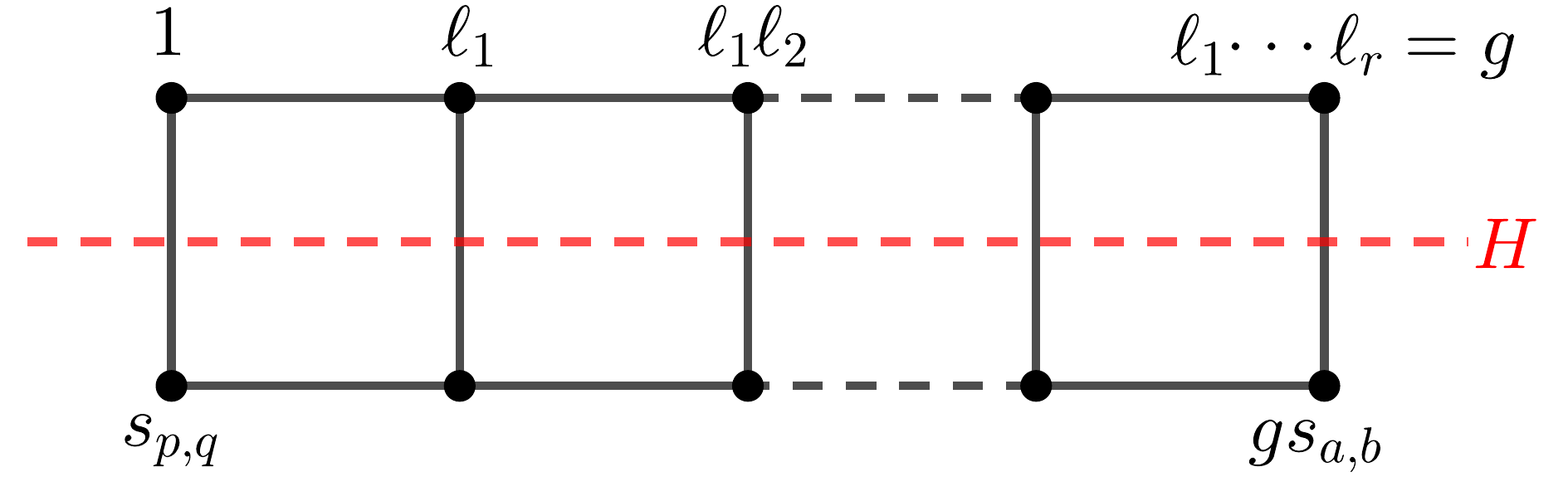}
\end{center}
As a consequence of the characterisation of the $4$-cycles in $\mathscr{C}_n$ given by Corollary~\ref{cor:Cycles}, this amounts to saying that $s_{p,q}$ can be shifted to the end of the word $s_{p,q}\ell_1 \cdots \ell_r$ by using mock commutations. During the process, $s_{p,q}$ becomes $s_{a,b}$; but, because $q-p+1$ is as small as possible, the letters $\ell_i$ remain unchanged. Hence $s_{p,q}g=gs_{a,b}$, which implies that 
$$s_{p,q} \cdot e = s_{p,q} \cdot (g,gs_{a,b})= (gs_{a,b}, gs_{a,b}s_{a,b}) = (gs_{a,b},g).$$
In other words, $s_{p,q}$ stabilises and inverts $e$, as desired. 
\end{proof}

\begin{cor}\label{cor:SplitReflection}
Fix $n \geq 2$ and $S \subset \{2, \ldots, n\}$. Let $R_n^S$ denote the subgroup of $J_n^S$ generated by the conjugates of all the generators $s_{p,q}$ satisfying $q-p+1= \min(S)$. Then $R_n^S$ is a right-angled Coxeter group and
$$J_n^S = R_n^S \rtimes J_n^{S \backslash \{ \min(S)\}}.$$
\end{cor}

\begin{proof}
Thanks to Lemma~\ref{lem:Reflection}, we can apply Theorem~\ref{thm:ReflectionGeneral} and deduce that $R_n^S$ is a right-angled Coxeter group and that $J_n^S$ decomposes as $R_n^S \rtimes \mathrm{stab}(Y)$ where $Y \subset \mathscr{C}_n^S$ denotes the intersection of all the halfspaces which contain $1$ and are delimited by hyperplanes labelled by sets of $> \min(S)$ strands. In other words, $Y$ is the subgraph of $\mathscr{C}_n^S$ given by all the vertices that can be reached from $1$ by a path whose edges are labelled by generators $s_{p,q}$ satisfying $q-p+1 \in S \backslash \{ \min(S)\}$. This amounts to saying that $Y= \mathscr{C}_n^{S \backslash \{ \min(S)\}}$. Since the stabiliser of this subgraph clearly coincides with $J_n^{S \backslash \{ \min(S) \}}$, the desired conclusion follows. 
\end{proof}

\noindent
By applying iteratively Corollary~\ref{cor:SplitReflection}, one obtains the following decomposition of $J_n$ as an iterated semi-direct product of right-angled Coxeter groups:
$$\begin{array}{lcl} J_n & = & R_n^{[2,n]} \rtimes J_n^{[3,n]} \\ \\ & = & R_n^{[2,n]} \rtimes \left( R_n^{[3,n]} \rtimes J_n^{[4,n]} \right) \\ & \vdots & \\ & = & R_n^{[2,n]} \rtimes \left( R_n^{[3,n]} \rtimes \left( \cdots \rtimes R_n^{[n,n]} \right) \right) \end{array}$$
This is the \emph{bottom-to-top decomposition} of $J_n$. Notice that, contrary to the top-to-bottom decomposition, the right-angled Coxeter groups here are typically not finitely generated.

\section{Acylindrical hyperbolicity}

\noindent
As mentioned in the introduction, having an action on some hyperbolic space may be quite useful in order to deduce valuable information on our group. Even though median graphs of cactus groups are almost never hyperbolic themselves, there exist techniques to recognise typical hyperbolic behaviours in median graphs and to extract from them actions on actual hyperbolic spaces. See \cite{MR4057355} and references therein for more information. 

\medskip \noindent
We begin by proving that most cactus groups do not satisfy some of the strongest forms of hyperbolicity.

\begin{prop}\label{prop:NotHyp}
The following assertions hold:
\begin{itemize}
	\item For every $n \geq 6$, $J_n$ is not hyperbolic.
	\item For every $n \geq 9$, $J_n$ is not relatively hyperbolic.
	\item For every $n \geq 8$, $J_n$ is not cyclically hyperbolic.
\end{itemize}
\end{prop}

\begin{proof}
If $n \geq 6$, then $J_n$ clearly contains a subgroup isomorphic to $J_3 \times J_3$, which is isomorphic to the direct sum of two infinite dihedral groups. Because a hyperbolic group cannot contain a subgroup isomorphic to $\mathbb{Z}^2$, we conclude that $J_n$ cannot be hyperbolic.

\medskip \noindent
If $n \geq 8$, then $J_n$ clearly contains a subgroup isomorphic to $J_4^{\{2\}} \times J_4^{\{2\}}$, which is isomorphic to the square of the free product between a cyclic group of order two and the infinite dihedral group. Because a cyclically hyperbolic group acting properly and cocompactly on a median graph cannot contain a product of two non-abelian free groups \cite{CyclicHyp}, we conclude that $J_n$ cannot be cyclically hyperbolic.

\medskip \noindent
Finally, assume that $n \geq 8$ and assume that $J_n$ is hyperbolic relative to some collection of subgroups $\mathcal{H}$. Our goal is to show that the whole group $J_n$ belongs to $\mathcal{H}$.

\medskip \noindent
For all disjoint segments $A,B \subset \{1, \ldots, n\}$ of size $\geq 3$, we denote by $J(A,B)$ (resp.\ $J(A)$, $J(B)$) the subgroup of $J_n$ given by the braided-like pictures supported on $A \sqcup B$ (resp.\ $A$, $B$). Notice that such a group is isomorphic to $J_a \times J_b$ for some $a,b \geq 3$. It follows from \cite[Theorem~4.19]{MR2182268} that there exists some $H(A,B) \in \mathcal{H}$ such that $J(A,B) \subset H(A,B)$. Given another pair of disjoint segments of size at least three $C,D \subset \{1, \ldots, n\}$, either $C$ is disjoint from $A$ or $B$, in which case $J(C)$ lies in $J(C,D)$ and normalises $J(A)$ or $J(B)$; or $C$ intersects both $A$ and $B$, which implies that $D$ is disjoint from $A$ or $B$, so $J(D)$ lies in $J(C,D)$ and normalises $J(A)$ or $J(B)$. In any case, $H(C,D)$ contains an infinite subgroup normalising an infinite subgroup of $H(A,B)$. But we know from \cite[Theorem~1.4]{MR2182268} that $\mathcal{H}$ is an almost malnormal collection of subgroups, so we must have $H(A,B)= H(C,D)$. Thus, we have proved that there exists a subgroup $H \in \mathcal{H}$ such that $J(A,B) \subset H$ for all disjoint segments $A,B \subset \{1, \ldots, n\}$ of size at least three.

\medskip \noindent
Let $1 \leq p < q \leq n$. If $q-p+1 = 2$, then we can find two disjoint segments $A,B \subset \{1, \ldots, n\}$ of size three such that $[p,q] \subset A$. Then $s_{p,q} \in J(A) \subset J(A,B) \subset H$. If $q-p+1>2$, then we can find two disjoint segments $A,B \subset \{1, \ldots, n\}$ such that $B$ has size three and such that $A \subset [p,q]$ has size three or four and is stable under the central inversion along $[p,q]$. Then $s_{p,q}$ normalises the infinite subgroup $J(A)$ of $H$. Again because $\mathcal{H}$ is an almost malnormal collection of subgroups, this implies that $s_{p,q}$ belongs to $H$. Thus, we have proved that $H$ contains all the generators of $J_n$, hence $H=J_n$. We conclude that $J_n$ is not relatively hyperbolic (with respect to proper subgroups). 
\end{proof}

\noindent
In contrast with Proposition~\ref{prop:NotHyp}, it turns out that most cactus groups are acylindrically hyperbolic. 

\begin{thm}\label{thm:AcylHyp}
For every $n \geq 4$, $J_n$ is acylindrically hyperbolic.
\end{thm}

\noindent
As recorded in \cite{MR4057355}, there exist many criteria that can be used in order to deduce from an action on a median graph that a given group is acylindrically hyperbolic. We will deduce Theorem~\ref{thm:AcylHyp} from:

\begin{prop}\label{prop:AcylHyp}
Let $G$ be a group acting properly on a median graph $X$. If an element $g \in G$ admits an axis crossing two hyperplanes that are both transverse to only finitely many common hyperplanes, then $g$ is contracting. As a consequence, $G$ is acylindrically hyperbolic unless it is virtually cyclic. 
\end{prop}

\noindent
In a median graph $X$, an \emph{axis} of an isometry $g \in \mathrm{Isom}(X)$ is a bi-infinite geodesic line on which $g$ acts as a translation. Our isometry is \emph{contracting} if it admits an axis $\gamma$ such that, for some constant $D \geq 0$, the nearest-point projection on $\gamma$ of every ball disjoint from $\gamma$ has diameter $\leq D$. Loosely speaking, a contracting isometry has the same behaviour as isometries in hyperbolic spaces. 

\begin{proof}[Proof of Proposition~\ref{prop:AcylHyp}.]
The first assertion follows from \cite[Theorem~6.8]{MR4057355}. The second assertion follows from \cite{MR3849623} or \cite{MR3415065}. 
\end{proof}

\begin{proof}[Proof of Theorem~\ref{thm:AcylHyp}.]
Define the element $g \in J_n$ by the word
$$s_{1,2} s_{2,3} \cdots s_{n-1,n} \cdot s_{n-2,n-1} s_{n-3,n_2} \cdots s_{2,3},$$
and let $[1,g]$ be the path from $1$ to $g$ labelled by this word. The key observation is that, given any power of our word, no two consecutive letters agree and no mock commutation applies. This implies that $\gamma:= \bigcup_{k \in \mathbb{Z}} g^k \cdot [1,g]$ is a convex geodesic line on which $g$ acts as a translation. 

\medskip \noindent
Let $H$ denote the hyperplane containing the edge $(1,s_{1,2})$. The convexity of $\gamma$ implies that the hyperplanes crossing $\gamma$ are pairwise non-transverse. Consequently, a hyperplane $K$ transverse to both $H$ and $gH$ has to be transverse to all the hyperplanes crossing $\gamma$ between $(1,s_{1,2})$ and $(g,gs_{1,2})$. It follows from Lemma~\ref{lem:HypTransverseLabel} that $K$ is labelled by a set of strands which is disjoint or nested with all the $[i,i+1]$, $1 \leq i \leq n-1$. The only possibility is that $K$ is labelled by $[1,n]$. Consequently, $H$ and $gH$ are strongly separated in $\mathscr{C}_n^{[2,n-1]}$. 

\medskip \noindent
But every vertex in $\mathscr{C}_n$ lies at distance $\leq 1$ from $\mathscr{C}_n^{[2,n-1]}$. Indeed, every vertex $x$ of $\mathscr{C}_n$ either already belongs to $\mathscr{C}_n^{[2,n-1]}$; or can be written as $ys_{1,n}$ for some $y \in \mathscr{C}_n^{[2,n-1]}$, because there is a mock commutation between $s_{1,n}$ and every other generator, in which case $x$ is adjacent to $y$. Thus, we deduce from our previous observation that $H$ and $gH$ are both transverse to only finitely many common hyperplanes. This proves that $g$ is a contracting isometry of $\mathscr{C}_n$.

\medskip \noindent
According to Proposition~\ref{prop:AcylHyp}, it only remains to check that $J_n$ is not virtually cyclic when $n \geq 4$. But $J_n$ contains $J_4^{\{2\}}$ as a subgroup, which decomposes as a free product between a cyclic group of order two and the infinite dihedral group. 
\end{proof}

\begin{cor}[\cite{Paolo}]\label{cor:Center}
For all $n \geq 3$ and $m \geq 4$, $PJ_m$ and $J_n$ have trivial centres.
\end{cor}

\begin{proof}
Notice that $J_3$ is an infinite dihedral group, so its centre is indeed trivial. Now, fix an $n \geq 4$. According to Theorem~\ref{thm:AcylHyp}, $PJ_n$ is acylindrically hyperbolic so its centre must be finite. Since $PJ_n$ is torsion-free according to Corollary~\ref{cor:PureTorsionFree}, it follows that the centre of $PJ_n$ is indeed trivial. Because $\Sigma : J_n \to S_n$ sends the centre of $J_n$ to centre of $S_n$, which is trivial, it follows that the centre of $J_n$ lies in $PJ_n$, and we conclude from our previous observation that the centre of $J_n$ must be trivial as well. 
\end{proof}

\begin{cor}\label{cor:NoFiniteNormal}
Fix an $n \geq 3$. Every finite normal subgroup in $J_n$ is trivial.
\end{cor}

\begin{proof}
If $n =3$, $J_n$ is an infinite dihedral group so the conclusion is clear. From now on, assume that $n \geq 4$. Fix a finite normal subgroup $N \lhd J_n$ and let $N'$ denote the intersection $N \cap J_n^{[2,n-1]}$. We claim that $N'$ is trivial. Notice that this is sufficient to conclude our proof. Indeed, since $J_n$ decomposes as $J_n^{[2,n-1]} \rtimes J_n^{\{n\}}$ and since $J_n^{\{n\}}$ is cyclic of order two, our claim implies that $N$ is either trivial or cyclic of order two. But a normal subgroup of order two must lie in the centre of the group, so we conclude from Corollary~\ref{cor:Center} that $N$ is trivial. 

\medskip \noindent
So let us prove that $N'$ is trivial. In the rest of the argument, we assume the reader familiar with Roller boundaries of median graphs. As constructed during the proof of Theorem~\ref{thm:AcylHyp}, there exists an element $g \in J_n^{[2,n-1]}$ acting on $\mathscr{C}_n^{[2,n-1]}$ with an axis $\gamma$ that is convex and that crosses infinitely many pairwise \emph{strongly separated} hyperplanes. (Two hyperplanes being \emph{strongly separated} if they cannot be both transverse to the same hyperplane.) As a consequence, the points $\gamma(\pm \infty)$ in the Roller boundary of $\mathscr{C}_n^{[2,n-1]}$ are isolated vertices. These two vertices at infinity are fixed by $N'$. Indeed, because $N'$ is finite, it stabilises some cube, say $Q$. Then, because $N'$ is normal, it also stabilises all the $g^kQ$, $k \in \mathbb{Z}$, which justifies the assertion. Again because $N'$ is normal, we know that $N'$ fixes all the points in the $J_n^{[2,n-1]}$-orbits of $\gamma(\pm \infty)$. Then $N'$ fixes the vertex of $\mathscr{C}_n^{[2,n-1]}$ given by the median point of three pairwise distinct in the previous orbits. Since vertex-stabilisers are trivial, we conclude that $N'$ must be trivial. 
\end{proof}

\section{Manifold structure}\label{section:manifolds}

\noindent
Interestingly, the cube-completion $\mathscr{A}_n$ of the median graph $\mathscr{C}_n$ on which the cactus group $J_n$ acts is almost a manifold. To be precise, the geometric side of the decomposition $J_n= J_n^{[2,n-1]} \rtimes J_n^{\{n\}}$, where $J_n^{\{n\}}$ is just a cyclic group of order two, is the decomposition of $\mathscr{A}_n$ as the product of the cube-completion $\mathscr{M}_n$ of $\mathscr{C}_n^{[2,n-1]}$ with (the cube-completion of) $\mathscr{C}_n^{\{n\}}$, which is just a single edge. Then:

\begin{prop}\label{prop:Manifold}
Fix an $n \geq 3$. The cube-completion $\mathscr{M}_n$ of the median graph $\mathscr{C}_n^{[2,n-1]}$ is an $(n-2)$-manifold. 
\end{prop}

\begin{proof}
It suffices to verify that the link $\mathfrak{L}_n$ of a vertex in $\mathscr{M}_n$ is topologically an $(n-3)$-sphere. Let $\mathfrak{L}_n^+$ denote the simplicial complex whose vertices are the intervals in $\{1, \ldots, n\}$ and whose edges connect two intervals whenever they are either disjoint or nested. Here, an interval has always length at least two. Then $\mathfrak{L}_n$ coincides with the subcomplex of $\mathfrak{L}_n^+$ spanned by the proper intervals of $\{1, \ldots, n\}$. Notice that $\mathfrak{L}_n^+$ is a cone over $\mathfrak{L}_n$, the apex corresponding to the vertex $[1,n]$.

\begin{claim}\label{claim:Manifold}
The complex $\mathfrak{L}_n$ is an $(n-3)$-manifold.
\end{claim}

\noindent
Let $\mathfrak{S} \subset \mathfrak{L}_n$ be a $k$-simplex for some $k \leq n-3$. The vertices of $\mathfrak{S}$ give $k+1$ pairwise disjoint or nested intervals in $\{1, \ldots, n\}$. Let $T$ be the rooted forest whose vertices are the singleton in $\{1, \ldots, n\}$ and the intervals given by $\mathfrak{S}$ such that the descendants of an interval are the intervals and singletons it contains. Observe that a vertex that is not a leaf has always at least two children and that there are at least two roots since our intervals are proper. 

\medskip \noindent
If $T$ contains a vertex $v$ which at least three children, then we can find three consecutive intervals $A,B,C \subset \{1, \ldots, n\}$ representing children of $v$. Then $\mathfrak{S}$ is contained in the two $(k+1)$-simplices spanned by $\mathfrak{S}\cup \{A \cup B\}$ and $\mathfrak{S} \cup \{B,C\}$. A similar argument holds if $T$ has at least three roots. We deduce from this observation two consequences. First, every $(n-4)$-simplex is contained in at least two $(n-3)$-simplices. And next, by iterating the argument, we find that $\mathfrak{S}$ is contained in a simplex whose forest is \emph{good}, i.e.\ which has two roots and all of whose vertices that are not leaves have two children. Since such a forest has $n$ leaves by construction, the forest has to contain exactly $n-2$ vertices that are not leaves, which amounts to saying that $\mathfrak{S}$ lies in an $(n-3)$-simplex. 

\medskip \noindent
It remains to verify that, assuming that $\mathfrak{S}$ is $(n-4)$-dimensional, it cannot be contained in three $(n-3)$-simplices. Fix an $(n-3)$-simplex $\mathfrak{S}^+$ containing $\mathfrak{S}$. As a consequence of the previous paragraph, we know that the forest $T^+$ of $\mathfrak{S}^+$ is good. Because $\mathfrak{S}$ can be obtained from $\mathfrak{S}^+$ by removing a vertex, the forest $T$ is obtained from $T^+$ either by removing a root or by removing a vertex $v$ and connecting the children of $v$ to the parent $v^+$ of $v$. In order to get a good forest from $T$ by adding a single vertex, there are only two possibilities: in the former case, we connect the new vertex to two roots given by consecutive intervals or singletons; in the latter case, $v^+$ has three children and we insert the new vertex between $v^+$ and two children given by consecutive intervals or singletons. Thus, $\mathfrak{S}$ is contained in exactly two $(n-3)$-simplices. This concludes the proof of Claim~\ref{claim:Manifold}. 

\medskip \noindent
Clearly, $\mathfrak{L}_3$ is a $0$-sphere (i.e.\ two isolated vertices). From now on, fix an $n \geq 4$ and assume that $\mathfrak{L}_m$ is an $(m-3)$-sphere for every $m <n$. 

\medskip \noindent
Identifying $\mathfrak{L}_{n-1}^+$ with the subcomplex of $\mathfrak{L}_n$ given by the intervals in $\{1, \ldots, n-1\}$, we can reconstruct $\mathfrak{L}_n$ by adding successively the vertices $[n-1,n]$, $[n-2,n]$, and so on. 

\medskip \noindent
Formally, given an $i \leq n-3$, let $\mathfrak{A}_{n-1}^i$ denote the subcomplex of $\mathfrak{L}_n$ spanned by $\mathfrak{L}_{n-1}^+$ and the vertices $[n-i,n], \ldots, n-1,n]$. Observe that each $\mathfrak{A}_{n-1}^{i}$ is obtained from $\mathfrak{A}_{n-1}^{i-1}$ by gluing a cone (whose apex is given by the vertex $[n-i,n]$) over the subcomplex spanned by the intervals disjoint from $[n-i,n]$ or proprely nested in it. The latter subcomplex is isomorphic to the join of $\mathfrak{L}_{n-i-1}^+$ and $\mathfrak{L}_{i+1}$. By induction, we know that $\mathfrak{L}_{n-i-1}$ is an $(n-i-4)$-sphere, which implies that $\mathfrak{L}_{n-i-4}^+$ is an $(n-i-3)$-ball, and that $\mathfrak{L}_{i+1}$ is an $(i-2)$ sphere, so our join is an $(n-4)$-ball. Thus, the complex $\mathfrak{A}_{n-1}^{n-3}$, which coincides with the complex obtained from $\mathfrak{L}_n$ by removing all the simplices containing $[2,n]$, is a manifold with boundary obtained from the $(n-3)$-ball $\mathfrak{L}_{n-1}^+$ by gluing successively cones over $(n-4)$-balls. This implies that $\mathfrak{A}_{n-1}^{n-3}$ is an $(n-3)$-ball. 

\medskip \noindent
In order to obtain $\mathfrak{L}_n$ from $\mathfrak{A}_{n-1}^{n-3}$, we need to glue a last cone (whose apex is given by $[n-1,n]$) over the subcomplex spanned by the intervals properly contained in $\{2, \ldots, n\}$. The latter subcomplex is isomorphic $\mathfrak{L}_{n-1}$, so this is an $(n-4)$-sphere. Thus, $\mathfrak{L}_n$ is a manifold obtained from an $(n-3)$-ball by gluing a cone over an $(n-4)$-sphere. This implies that $\mathfrak{L}_n$ must be an $(n-3)$-sphere, as desired. 
\end{proof}

\noindent
As a first application of Proposition~\ref{prop:Manifold}, let us observe that the cactus group $J_4$ is virtually a surface group.

\begin{cor}\label{cor:Surface}
The pure cactus group $PJ_4$ is a hyperbolic surface group.
\end{cor}

\begin{proof}
According to Theorem~\ref{thm:conspicial}, $PJ_4$ acts freely and cocompactly on $\mathfrak{A}_n$, and it follows from Proposition~\ref{prop:Manifold} that $\mathfrak{A}_4$ decomposes as $\mathfrak{M}_4 \times [0,1]$ where $\mathfrak{M}_4$ is a surface. By collapsing $[0,1]$, we conclude that $PJ_4$ acts freely and cocompactly on a simply connected surface, so it has to be the fundamental group of a surface. Moreover, the latter surface cannot be a sphere, a torus, a projective plane, nor a Klein bottle since $PJ_4$ contains a non-abelian free subgroup. 
\end{proof}

\noindent
It worth mentioning that the compactifications of modular spaces whose fundamental groups are classically identified with pure cactus groups are also aspherical manifolds, so the fact that the pure cactus group $PJ_n$ can be described as the fundamental group of an aspherical $(n-2)$-dimensional is not new. This also includes Corollary~\ref{cor:Surface}. It can be verified that $PJ_4$ is the fundamental group of the non-orientable surface of Euler characteristic $-3$. However, the geometry provided by Proposition~\ref{prop:Manifold} will allow us to deduce more information, namely Corollary~\ref{cor:QI} below. 

\medskip \noindent
As a second application of Proposition~\ref{prop:Manifold}, it is possible distinguish cactus groups from different perspectives. Up to isomorphism, cactus groups are easy to distinguish: it suffices to compute their abelianisations from their presentations, namely $J_n / [J_n,J_n] \simeq (\mathbb{Z}/2\mathbb{Z})^{n-1}$ for every $n \geq 2$. However, the same question for pure cactus groups is already more difficult. However, by cohomological arguments, one easily deduces from Proposition~\ref{prop:Manifold} that pure cactus groups are pairwise non-isomorphic, and even that cactus groups are pairwise non-commensurable. By using coarse cohomology, it is possible, more generality, to distinguish cactus groups up to quasi-isometry.

\begin{cor}\label{cor:QI}
For all $n,m \geq 2$, the cactus groups $J_n$ and $J_m$ are quasi-isometric if and only if $n=m$.
\end{cor}

\noindent
In order to prove the assertion, we need the following observation:

\begin{fact}\label{fact:UC}
Let $X$ be a median graph. If its cube-completion $X^\square$ is finite-dimensional, then it is \emph{uniformly contractible}, i.e.\ for every $R \geq 0$, there exists some $S \geq R$ such that every ball of radius $R$ is contractible in a ball of radius $S$.
\end{fact}

\begin{proof}[Proof of Fact~\ref{fact:UC}.]
Cube-completions of median graphs can be endowed with a CAT(0) metric, biLipschitz equivalent to the combinatorial metric on the one-skeleton. Since balls in CAT(0) spaces are convex, the desired conclusion follows.
\end{proof}

\begin{proof}[Proof of Corollary~\ref{cor:QI}.]
If $J_n$ and $J_m$ are quasi-isometric, then so are the manifolds $\mathscr{M}_n$ and $\mathscr{M}_m$, which implies that the coarse cohomology groups of $\mathscr{M}_n$ and $\mathscr{M}_m$, as defined in \cite{MR2007488}, agree. But we know from \cite[Theorem~5.28]{MR2007488}, which applies thanks to Fact~\ref{fact:UC}, that coarse cohomology groups and compactly supported cohomology groups of our manifolds agree. Consequently, $\mathscr{M}_n$ and $\mathscr{M}_m$ must have the same dimension, i.e.\ $n=m$. 
\end{proof}

\begin{remark}
As an alternative proof of Corollary~\ref{cor:QI}, one can verify that the asymptotic dimension of the cactus group $J_n$ equals $n-2$. Indeed, as proved in \cite[Paragraph~$1.F_1$]{MR1253544}, the asymptotic dimension of a uniformly contractible manifold is bounded below by its dimension\footnote{In fact, \cite{MR1253544} claims that the two dimensions agree, but the reverse inequality is disproved by \cite[Example~3.4]{MR1607744}}. On the other hand, we know from \cite{MR2916293} that the asymptotic dimension of the cube-completion of a median graph is bounded above by its dimension. 
\end{remark}

\section{Open problems}\label{section:Open}

\noindent
In this final section, we record various problems and questions about cactus groups.

\paragraph{Computational problems.} We saw in Section~\ref{section:WandCP} that the word and conjugacy problems are (explicitly and efficiently) solvable in cactus groups. But there are other natural computational problems to look at. For instance:

\begin{problem}
Solve explicitly and efficiently the order problem in cactus groups.
\end{problem}

\noindent
In other words, we would like a nice algorithm that determines the order of a given element. It is not difficult to deduce from the median geometry described in Section~\ref{section:MedianGeometry} that the order problem is solvable\footnote{Let $g \in J_n$ be an element. Up to replacing $g$ with a large power (depending only on $n$), we can assume that, as an isometry of $\mathscr{C}_n$, no power of $g$ inverts a hyperplane. For every vertex $x \in \mathscr{C}_n$, we can check whether it is fixed by $g$ or belongs to an axis of $g$ (by checking whether $x, gx, \ldots, g^kx$ all lie on common geodesic for some $k$ large enough compared to $n$). In the former case, $g$ has finite order (which can be computed since the word problem is solvable); and in the latter case, $g$ has infinite order. We check the vertices of larger and larger balls around $1$.}, but some additional work is required in order to solve the problem in a more satisfying way, in terms of braided-like pictures or interval diagrams. 

\begin{problem}
Describe centralisers in cactus groups. As a consequence, solve the commutation problem.
\end{problem}

\noindent
Again, the median geometry of cactus groups can be helpful in order to solve this problem. Either by exploiting the structure of group presentation diagrams of minimal area as in Section~\ref{section:Conjugacy}, or by using the structure of centralisers of isometries of median graphs (see for instance \cite{MedianSets}). 

\begin{problem}
Solve explicitly and efficiently the power problem in cactus groups.
\end{problem}

\noindent
In other words, we would like an algorithm that determines, for every element $g$ of our cactus group $J_n$, the (finite) set $\{ (x,k) \in J_n \times \mathbb{N} \mid x^k=g\}$. 

\medskip \noindent
Another important computational problem is the isomorphism problem. However, as discussed in Section~\ref{section:manifolds}, cactus groups can be easily distinguished. Nevertheless, the following related problem is of interest:

\begin{problem}
Describe the automorphism groups of cactus groups.
\end{problem}

\paragraph{Subgroups of cactus groups.} We saw with Corollary~\ref{cor:Torsion} that finite subgroups in cactus groups are $2$-groups, but having a precise description of (maximal) finite subgroups in cactus groups would be more satisfying. Since finite groups of isometries in median graphs stabilise cubes, it is possible to exploit the median geometry describe in Section~\ref{section:MedianGeometry} in order to find a description of finite subgroups of cactus groups as subgroups of symmetry groups. However, their algebraic structures require some additional work in order to be fully understood. We conjecture that maximal finite subgroups in cactus groups decompose as iterated semi-direct products of powers of $\mathbb{Z}/2\mathbb{Z}$.

\medskip \noindent
More precisely, given an arbitrary group $A$ and an integer $r \geq 1$, let $A \wr_{r} \mathbb{Z}/2\mathbb{Z}$ denote the permutational wreath product $A^r \rtimes \mathbb{Z}/2\mathbb{Z}$ where the right factor acts on the direct sum by sending the $i$th factor to the $(r-i+1)$th factor. Then:

\begin{conj}
A finite group embeds into the cactus group $J_n$ if and only if it embeds into 
$$( \cdots (\mathbb{Z}/2\mathbb{Z} \wr_{r_1} \mathbb{Z}/2\mathbb{Z}) \wr_{r_2} \cdots ) \wr_{r_s} \mathbb{Z}/2\mathbb{Z}$$
for some $r_1, \ldots, r_s$ satisfying $r_1 \cdots r_s \leq n/2$. 
\end{conj}

\noindent
Another natural family of subgroups to investigate is given by free abelian groups. Can we compute the \emph{algebraic dimensions} of cactus groups, i.e.\ the maximal ranks of free abelian subgroups? We can easily construct a free abelian subgroup of rank $\lfloor n/3 \rfloor$ in $J_n$. We conjecture that we cannot do better.

\begin{conj}\label{conj:AlgebraicDim}
The cactus group $J_n$ contains a free abelian subgroup of rank $k$ if and only if $k \leq \lfloor n/3 \rfloor$. 
\end{conj}

\noindent
It is worth noticing that this conjecture can be rephrased purely geometrically: the median graph $\mathscr{C}_n$ contains an isometrically embedded copy of the canonical Cayley graph of $\mathbb{Z}^k$ if and only if $k \leq \lfloor n/3 \rfloor$. One direction of the equivalence between our two assertions is given by \cite{MR3625111} or \cite{MR3704240}, and the other direction is given by \cite{SpecialHyp} (see also \cite{MR4242153}), which applies thanks to Theorem~\ref{thm:conspicial}. We expect this geometric formulation of Conjecture~\ref{conj:AlgebraicDim} to be more accessible. 

\medskip \noindent
We already mentioned that the algebraic dimension of our cactus group $J_n$ is bounded below by $\lfloor n/3 \rfloor$. As a consequence of the previous discussion, it is also bounded above by the dimension of the cube-completion of $\mathscr{C}_n$, which coincides with the maximal number of pairwise disjoint or nested intervals in $\{1, \ldots, n\}$, i.e.\ $n-1$. 

\medskip \noindent
Characterising cactus groups with no $\mathbb{Z}^2$ is related (and in fact equivalent) to hyperbolicity. As mentioned in the proof of Proposition~\ref{prop:NotHyp}, the cactus group $J_n$ contains $\mathbb{Z}^2$ as soon as $n \geq 6$. But what can be said about $J_n$ for $n =4, 5$? We saw with Corollary~\ref{cor:Surface} that $J_4$ is virtually a surface group, but what about $J_5$?

\begin{question}
What is the structure of the cactus group $J_5$? of its pure subgroup $PJ_5$?
\end{question}

\noindent
Notice that it follows from Proposition~\ref{prop:Manifold} that $PJ_5$ is the fundamental group of a compact hyperbolic $3$-manifold. In particular, $PJ_5$ cannot be a surface group.

\paragraph{Homology groups.} Because cube-completion of median graphs are contractible and since pure cactus groups act on their median graphs with trivial cube-stabilisers, we obtain compact classifying spaces of pure cactus groups. More precisely, fix an $n \geq 2$ and let $\mathscr{S}_n$ denote the cube complex
\begin{itemize}
	\item whose vertices are the permutations of $\{1, \ldots, n\}$;
	\item whose vertices connect two permutations $\mu,\nu$ whenever $\mu = \iota \nu$ for some central inversion $\iota$ of an interval in $\{1, \ldots, n\}$;
	\item whose $k$-cubes are spanned by edges of the forms $(\sigma,\iota_1 \sigma), \ldots, (\sigma, \iota_k \sigma)$ where $\iota_1, \ldots, \iota_k$ are central inversions of pairwise disjoint or nested intervals in $\{1, \ldots, n\}$.
\end{itemize}
Because $\mathscr{S}_n$ coincides with the quotient modulo $PJ_n$ of the cube-completion of $\mathscr{C}_n$, it defines a classifying space for $PJ_n$. Consequently, $\mathscr{S}_n$ can be useful in order to compute various algebraic invariants. For instance, is it possible to use discrete Morse theory or a collapsing scheme in order to compute homology groups of pure cactus groups? This could allow us to recover some of the results obtained in \cite{MR2630055} by using different methods. As another possible application:

\begin{problem}
Describe group presentations of pure cactus groups.
\end{problem}

\noindent
In particular, this would allow us to compute abelianisations of pure cactus groups.

\paragraph{Generalisation.} Mimicking the definition of cactus groups, one can associate a generalised cactus group to every Coxeter group. Given a Coxeter system $(W,S)$, a non-empty subset $I \subset S$ is \emph{connected} if the subgroup $W_I \subset W$ is finite and does not decompose as $W_J \times W_K$ for some disjoint non-empty subsets $J,K \subset S$. The \emph{generalised cactus group} $C_W$ is then defined by generators $s_I$, $I \subset S$ connected, satisfying the relations:
\begin{itemize}
	\item $s_I^2=1$ for every $I \subset S$ connected;
	\item $s_Is_J = s_J s_{w_J(I)}$ for all connected $I,J \subset S$ satisfying $I \subset J$ or $W_{I \cup J} = W_I \times W_J$, where $w_J$ denotes the longest element in $W_J$.
\end{itemize}
The group $C_W$ naturally surjects onto the Coxeter group $W$, and the kernel of the morphism is referred to as the \emph{generalised pure cactus group} $PC_W$. Cactus groups correspond to generalised cactus groups associated to symmetric groups. 

\medskip \noindent
It is natural to ask whether the geometric perspective on cactus groups described in this article generalises. 

\begin{problem}
Describe the median geometry of generalised cactus groups.
\end{problem}

\noindent
It is worth mentioning that generalised pure cactus groups also embed into right-angled Coxeter groups \cite{GenCactus}, so we already know that such a median geometry exists (since canonical Cayley graphs of right-angled Coxeter groups are median). The problem is to determine whether it is as useful as for cactus groups. As a concrete question in this direction:

\begin{problem}
Solve explicitly and efficiently the conjugacy problem in generalised pure cactus groups. 
\end{problem}

\appendix

\section{Embeddings into right-angled Coxeter groups}

\noindent
This section is dedicated to the proof of Theorem~\ref{thm:Special}. A similar statement for right-angled Artin groups can be found in \cite{MR4242153}. Our proof follows the point of view introduced in \cite{QMspecial}. 

\medskip \noindent
First of all, recall that, given a (simple) graph $\Gamma$, a word of generators $w:= \ell_1 \cdots \ell_r$ is \emph{$\Gamma$-reduced} if there do not exist $1 \leq i< j \leq r$ such that $\ell_i=\ell_j$ and such that this common letter is adjacent in $\Gamma$ to all the letter $\ell_k$ for $i< k < j$. Any element of the right-angled Coxeter group $C(\Gamma)$ can be represented by a $\Gamma$-reduced word, which is unique up to permuting commuting letters. In particular, a $\Gamma$-reduced word represents the trivial element in $C(\Gamma)$ if and only if it is empty. 

\medskip \noindent
Now, let $G$ be a group acting on a median graph $X$. Fix colouring $\mathfrak{c}$ of the hyperplanes of $X$ and $\Gamma$ a graph whose vertices are the possible colours of a hyperplane. Assume that:
\begin{itemize}
	\item no two hyperplanes of the same colour are transverse;
	\item the colours of any two transverse hyperplanes are adjacent in $\Gamma$;
	\item $\mathfrak{c}$ is $G$-equivariant, i.e.\ $\mathfrak{c}(gA)=\mathfrak{c}(A)$ for every $g \in G$ and every hyperplane~$A$.
\end{itemize}
Fix a basepoint $o \in X$ and define a map $G \to C(\Gamma)$ as  
$$\Omega : g \mapsto \begin{array}{c} \text{word of colours given by the hyperplanes} \\ \text{successively crossed by a path from $o$ to $go$} \end{array}.$$
Let us verify that the element of $C(\Gamma)$ we get does not depend on the path we choose. Because any two paths in a median graph with the same endpoints can be connected by a sequence of elementary operations, namely flipping $4$-cycles and adding or removing backtracks, it suffices to verify that these operations do not affect the element of $C(\Gamma)$ we get.

\medskip \noindent
So fix a path $[o,go]$ from $o$ to $go$ and let $w$ be the corresponding word of colours. Removing a backtrack from $[o,go]$ transforms $w$ by removing two consecutive letters having the same colour. This does not modify the element of $C(\Gamma)$. Adding a backtrack to $[o,go]$ transforms $w$ by adding somewhere two consecutive letters having the same colour. Again, this does not modify the element of $C(\Gamma)$. Finally, flipping a $4$-cycle transforms $w$ by switching two consecutive letters whose colours are adjacent in $\Gamma$. And this does not modify the element of $C(\Gamma)$ either.

\medskip \noindent
Thus, $\Omega$ is well-defined. This is moreover a morphism. Indeed, for all $g,h \in G$, the word of colours labelling a path $[o,gh \cdot o]$ which is the concatenation of a path $[o,go]$ with a path $g [o,ho]$ is the concatenation of the word labelling $[o,go]$ with the word labelling $[o,ho]$ (which is also the word labelling $g [o,ho]$ since $\mathfrak{c}$ is $G$-equivariant), hence the equality $\Omega(gh)=\Omega(g) \Omega(h)$. 

\medskip \noindent
In general, the morphism $\Omega$ may be trivial. We need some extra-assumptions in order to make it interesting. From now on, we assume that:
\begin{itemize}
	\item any two hyperplanes whose carriers intersect are equal if and only if they have the same colours;
	\item any two hyperplanes whose carriers intersect are transverse if and only if their colours are adjacent in $\Gamma$.
\end{itemize}
We claim that now $\Omega$ must be injective. More precisely, we want to prove that, given an element $g \in G$, the word of colours labelling a geodesic $[o,go]$ is $\Gamma$-reduced. Otherwise, there exist two edges $e_1,e_2 \subset [o,go]$ such that the hyperplanes containing $e_1,e_2$ have the same colour and such that this colour is adjacent in $\Gamma$ with the colours of all the hyperplanes crossing $[o,go]$ between $e_1,e_2$. Consider the edge $e$ adjacent to $e_2$ and lying between $e_1,e_2$ along $[o,go]$, if such an edge exists. Then it follows from our assumptions that the hyperplanes containing $e$ and $e_2$ are transverse, which amounts to saying that $e$ and $e_2$ span a $4$-cycle. Thus, by flipping a $4$-cycle we can make $e_2$ closer to $e_1$. By iterating this process, we can reduce to the case where $e_1$ and $e_2$ are adjacent along our geodesic $[o,go]$. It follows from our assumptions that the hyperplanes containing $e_1$ and $e_2$ are identical, which amounts to saying that $e_1=e_2$. But this is impossible since $[o,go]$, as a geodesic, cannot contain a backtrack. This concludes the proof of our claim, and of Theorem~\ref{thm:Special}.

\addcontentsline{toc}{section}{References}

\footnotesize
\bibliographystyle{alpha}
\bibliography{Cactus}

\end{document}